\documentclass[reqno,12pt,letterpaper]{amsart}
\usepackage{amsmath,amssymb,amsthm,graphicx,mathrsfs,url}
\usepackage[usenames,dvipsnames]{color}
\usepackage[colorlinks=true,linkcolor=Red,citecolor=Green]{hyperref}
\usepackage{amsxtra}
\usepackage{enumitem}
\usepackage{epstopdf}

\usepackage{stix}
\usepackage[dvipsnames]{xcolor}

\makeatletter
\def\Ddots{\mathinner{\mkern1mu\raise\p@
\vbox{\kern7\p@\hbox{.}}\mkern2mu
\raise4\p@\hbox{.}\mkern2mu\raise7\p@\hbox{.}\mkern1mu}}
\makeatother
\usepackage{graphicx}

\def\?[#1]{\textbf{[#1]}\marginpar{\Large{\textbf{??}}}}

\let\epsilon=\varepsilon 
\newcommand{\dd}{\, {\rm d}}

\newtheorem{theorem}{Theorem}
\newtheorem*{theorem*}{Theorem}
\newtheorem{proposition}{Proposition}[section]

\newtheorem{lemma}{Lemma}
\newtheorem{corr}{Corollary}
\newtheorem{remark}{Remark}

\numberwithin{equation}{section}

\DeclareMathOperator{\sgn}{sgn}

\title[Stability of Rayleigh-Jeans equilibria in the kinetic FPU equation]{Stability of Rayleigh-Jeans equilibria \\ in the kinetic FPU equation}

\author{Pierre Germain} 
\address{Department of Mathematics, Huxley building, South Kensington campus, Imperial College London, London SW7 2AZ, United Kingdom}
\email{pgermain@ic.ac.uk}

\author{Joonhyun La} 
\address{June E Huh Center for Mathematical Challenges, 85 Hoegi-ro, Dongdaemun-gu, Seoul 02455, Republic of Korea.}
\email{joonhyun@kias.re.kr}

\author{Angeliki Menegaki} 
\address{Department of Mathematics, Huxley building, South Kensington campus, Imperial College London, London SW7 2AZ, United Kingdom}
\email{a.menegaki@imperial.ac.uk}

\begin{document}
\maketitle

\begin{abstract} We study the nonlinear dynamics of the kinetic wave equation associated to the FPU problem and prove stability of the non-singular Rayleigh-Jeans equilibria. The lack of a spectral gap for the linearized problem leads to polynomial decay, which we are able to leverage to obtain nonlinear stability.
\end{abstract}

\tableofcontents

\section{Introduction}

\subsection{The microscopic model}

In their groundbreaking 1955 study, \cite{Fermi1955StudiesON}, Fermi, Pasta, Ulam and Tsingou utilized the early electronic computers to explore the relaxation dynamics (thermalisation) of a chain of coupled nonlinear oscillators. 
The Fermi-Pasta-Ulam-Tsingou system is an anharmonic chain of oscillators without pinning potential. The system is described by its momentum $(p_n)_{n \in \mathbb{Z}}$ and position $(q_n)_{n \in \mathbb{Z}}$ with the Hamiltonian energy given by
 $$H(p,q) = \sum_{n \in \mathbb{Z}} \frac{p_n^2}{2} + \sum_{n \in \mathbb{Z}} F(q_{n+1} - q_n)$$
 where $F$ is smooth and satisfies $F(0) = 0$.

The dynamics are governed by the system of ODEs

$$
-\ddot{q_n} = F''(q_{n+1}-q_n) - F''(q_n - q_{n-1})
$$
A particular example is the so-called FPUT-$\beta$ case
$$
F(x) = \frac{1}{2} x^2 + \frac{\beta}{4} x^4;
$$
it will be the focus of the present article.

Motivated by the original FPUT problem and its substantial impact \cite{fiftyFPU}, significant attention has been given to the energy transport within oscillator chains of length $N$ coupled to thermal reservoirs at different temperatures. Numerically, it is observed that in $\beta$-FPUT chains, the thermal conductivity $\kappa_N$ diverges as $N$ increases in an anomalous way, specifically as $\kappa_N \sim N^{2/5}$, \cite{AokiKuzn01, LLP05}.
An alternative approach is to consider an infinite chain and observe energy spread after injecting energy at the origin. This reformulates the problem to studying the time-decay of the energy current-current correlation function, which is related to the decay rate of the linearized semigroup of the wave turbulence kinetic equation arising from such microscopic models \cite{AokiLukkSpohn}. For the linearised problem of the $\beta$-FPUT chain, it was proven in \cite{LukkarinenSpohn2008} that the correlation decays as $\mathcal{O}(t^{-3/5})$, aligning with numerical findings. Following the study in \cite{LukkarinenSpohn2008}, in \cite{MelletMerino} it was rigorously derived a macroscopic fractional diffusion equation describing heat transport in $\beta$-FPUT chains, confirming the anomalous diffusion behavior from the linearized Boltzmann phonon equation. Our main inspiration and motivation here is also the findings in \cite{LukkarinenSpohn2008}.

We also mention other classical choices of nonlinear atom chains. Beside the $\beta$-FPU chain, one may consider the $\alpha$-FPU chain, where $F(x) = \frac{x^2}{2} + \alpha \frac{x^3}{3}$, or the $\alpha+\beta$-FPU chain, which combines the nonlinearities of both the $\alpha$ and $\beta$ chains. Moreover, adding an additional pinning potential to the dynamics results in the Discrete Nonlinear Klein-Gordon chain, where $F_{KG}(x) = \frac{\delta}{2}x^2$ and the energy includes $U_{pin}(q) = (\frac{1}{2}-\delta) q^2 + \frac{1}{4} q^4$. Finally, the Toda lattice is yet another model, where $F_T(x) = \frac{1}{4\alpha^2}(e^{2\alpha x}-1-2\alpha x)$. 
For a detailed account and discussion on these models, we refer to the reviews \cite{OLDC2023, Lukkarinen2016, Spohnreview06}.


\subsection{The homogeneous kinetic wave equation}
In this article, we focus on the kinetic wave equation, or phonon Boltzmann equation, arising from the $\beta$-FPUT system.
The kinetic wave equation we will consider can be written 
$$
\partial_t f(t,p) = \mathcal{C}[f](t,p)
$$
where the collision operator is given by
\begin{equation} 
\label{collisionop}
\mathcal{C}[f](t,p_0)=\int_{\mathbb{T}^{3}} \delta(\Sigma) \delta(\Omega) \prod_{\ell=0}^3 \omega_\ell \prod_{\ell=0}^3 f_\ell \left(\frac{1}{f}+\frac{1}{f_1}-\frac{1}{f_2}-\frac{1}{f_3}\right)\dd p_1 \dd p_2 \dd p_3
\end{equation} 
with the usual notations $p=p_0$, $f= f_0 = f(p_0)$, $f_i=f(p_i)$ if $i=1,2,3$, similarly for $\omega_i=\omega(p_i)$, and furthermore
\begin{align*}
& \omega(p) = \left| \sin \left(\frac{p}{2} \right) \right|\\
& \Sigma = \Sigma (p, p_1, p_2,p_3) = p_0+p_1-p_2-p_3 \\
& \Omega= \Omega (p, p_1, p_2,p_3) = \omega_0 + \omega_1 -\omega_2 - \omega_3
\end{align*}
and finally we denote $\mathbb{T}$ for the periodized torus
$$
\mathbb{T} = \mathbb{R} / 2\pi \mathbb{Z}.  
$$

The equation conserves two quantities which will play an important role in our analysis: the mass $\mathcal{M}$ and energy $\mathcal{E}$
\begin{align*}
& \mathcal{M}(f) = \int_{\mathbb{T}} f(p) \dd p \\
& \mathcal{E}(f) = \int_{\mathbb{T}} \omega(p) f(p) \dd p.
\end{align*}

The equation also satisfies an H-theorem
$$
\frac{d}{dt} \int \log f(p) \dd p \leq 0,
$$
even though we will not make use of this fact.

Finally, the Rayleigh-Jeans equilibria (RJ) 
$$
\mathfrak{f}_{\beta,\gamma}(p) = \frac{1}{\beta \omega(k) + \gamma}, \qquad \beta,\gamma \geq 0
$$
are the unique stationary solutions \cite[Section 5]{LukkarinenSpohn2008}. 

The case $\gamma=0$ corresponds to a singular RJ equilibrium, which has infinite mass. However, it is favored by some authors since it gives equipartition of energy, as was initially expected by E. Fermi et.al.  in their numerical experiment \cite{OLDC2023,Lukkarinen2016}.

Interestingly, not all couples $(\mathcal{M}_0,\mathcal{E}_0)$ can be associated to a RJ equilibrium. For initial data whose mass and energy cannot be matched to a RJ equilibrium, the asymptotic behavior of the solution is an intriguing question - see the discussion in Section \ref{MERJ}. 

\subsection{Making sense of the collision operator} The expression for the collision operator in \eqref{collisionop} is not obviously meaningful, even for smooth functions $f$. Indeed, it involves the product of two $\delta$ functions, which, as is well-known, can be ill-defined.

As was showed by in \cite{LukkarinenSpohn2008} (see also Lemma \ref{lemmaparameterization}), the common zeros of $\Omega$ and $\Sigma$ are of two kinds.

On the one hand, trivial zeros are such that $\{ p_0, p_1 \} = \{ p_2, p_3 \}$. The integrand of the collision operator vanishes on the set of trivial zeros, which provides some cancellation. However, from the viewpoint of distribution theory, it is not clear how to make sense of the product of two Dirac $\delta$, even if it is evaluated on a function vanishing on their singular set ("indefinite form"). Lukkarinen and Spohn \cite{LukkarinenSpohn2008}, in the linearized case, show that the contribution of trivial zeros vanishes if one resorts to a regularization procedure. This regularization is closely related to the microscopic derivation of the equation (which remains an outstanding open problem) and we shall take for granted that trivial zeros can be ignored in the collision operator.

On the other hand, non-trivial zeros can be parameterized by
$$
p_1 = h(p_0,p_2) \mod 2\pi, \quad \mbox{where} \quad  h(x,z) = \frac{z-x}{2} + 2 \arcsin \left( \tan \frac{|z-x|}{4} \cos \frac{z+x} 4 \right).
$$
Using this expression, a calculation (\cite{LukkarinenSpohn2008} or Lemma \ref{lemmaintegrationp2}) shows that the collision operator can be written 
\begin{equation}
\label{formulacollision}
\mathcal{C}[f](p_0)= \int_0^{2\pi}
\frac{ \omega_0 \omega_1 \omega_2 \omega_3
 }{\sqrt{F_+(p_0,p_2)}}  \prod_{\ell=0}^3 f_\ell \left(\frac{1}{f}+\frac{1}{f_1}-\frac{1}{f_2}-\frac{1}{f_3}\right)
 dp_2,
\end{equation}
where it is understood that
$$
\begin{cases}
& p_3 = p_0 + p_1 - p_2 \\
& p_1 = h(p_0,p_2).
\end{cases}
$$
and
$$
F_+(p_0,p_2) = \sqrt{\left[ \cos \left( \frac{p_0}{2}\right) +\cos \left( \frac{p_2}{2}\right) \right]^2 +4\sin\left( \frac{p_0}{2}\right)\sin\left( \frac{p_2}{2}\right)}.
$$

\subsection{Main results and organization of the paper}

\textit{Section \ref{sectionLWP}} establishes basic local well-posedness results in weighted $L^\infty$ spaces. It also shows that local well-posedness cannot be expected (or at least is very delicate) in $L^p$-based spaces for $p < 3$. The results of this section (theorems \ref{theo: Linfty lwp} and \ref{thmunbounded}) can be summarized as the following theorem.

\begin{theorem*}
The collision operator is bounded in $\omega^\alpha L^\infty$ for $\alpha \geq -1$. Thus, the equation is locally well-posed in $\mathcal{C} (\omega^\alpha L^\infty)$ (the space of continuous functions with values in $\omega^\alpha L^\infty$ for $\alpha \geq -1$).

The collision operator is unbounded on $L^p$ for any $p < 3$.
\end{theorem*}

In \textit{Section \ref{sectionfirstproperties}}, we turn to the linearized problem around RJ equilibria, recapitulating many results obtained in \cite{LukkarinenSpohn2008}, and providing some extensions. The traditional tools of kinetic theory  which we apply only yield an estimate of type $L^2_{t,p}$ for the perturbation (Corollary \ref{corodissi}):

\begin{theorem*} Denoting $L$ for the linearized operator around a RJ solution, there exists a function $a(p)$ such that $|a(p)| \sim |\sin \left( \frac p 2 \right)|^{\frac 53}$ and
$$
\int_0^\infty \int_0^{2\pi} a(p) |e^{tL} g_0(p)|^2 \dd p \dd t \lesssim \| g_0 \|_{L^2}^2.
$$
\end{theorem*}
This estimate is insufficient for nonlinear purposes: the decay in time is weak, and we cannot hope to close nonlinear estimates since the equation is ill-posed in $L^2$ type spaces. Furthermore, the lack of a spectral gap leads to the degenerate weight $a(p)$.

\textit{Section \ref{sectionpointwise}} addresses these shortcomings by proving pointwise decay at a polynomial rate (Theorem \ref{theoremdecay}):

\begin{theorem*}Assume that $g_0 \in L^\infty$, with a zero projection (in $L^2$) on $\operatorname{Ker} L$. Then for any $\mu ,\nu \in [\frac 16, \frac 12]$ and any $\delta>0$
$$
\| \omega^\mu e^{tL} g \|_{L^\infty} \lesssim_\delta \langle t \rangle^{- \frac{3}{5}(\mu + \nu)+ \delta} \| \omega^{-\nu} g_0 \|_{L^\infty}.
$$
\end{theorem*}

This is achieved by understanding how the edges of the frequency domain, where dissipation degenerates, interact with the bulk of the domain. At a more technical level, we resort to an iterative scheme to gain decay increasingly.

Finally, \textit{Section \ref{sectionnonlinear}} deals with the fully nonlinear problem, showing global existence of solutions around non-singular RJ. This is accomplished through a careful definition of an appropriate norm that allows us to control the nonlinearity, relying crucially on the structure of the equation (Theorem \ref{thmnonlinstab}).

\begin{theorem*} There exists $\epsilon_0$ such that the following holds. If
$$
f(t=0) = \mathfrak{f}_{\beta,\gamma} [ 1 + g_0]
$$
where $\int \mathfrak{f}_{\beta,\gamma} g_0 \dd p = \int \omega  \mathfrak{f}_{\beta,\gamma} g_0 \dd p = 0$ and $\| \omega^{-\frac 12} g_0 \|_{L^\infty} = \epsilon < \epsilon_0$,
then there exists a global solution which can be written
$$
f(t) = \mathfrak{f}_{\beta,\gamma} [ 1 + g] \qquad \mbox{where} \quad \| \omega^{\frac{1}{2}} g(t,p) \|_{L^\infty_p} \lesssim \epsilon \langle t \rangle^{-\frac 35 + \frac{1}{1000}} \quad \mbox{for all $t \geq 0$}.
$$
\end{theorem*}

The appendix gathers some elementary but involved computations which are used in the rest of the text.

\subsection{Stability of equilibria for the Boltzmann equation}
It is worthwhile comparing the results which have been stated above to  analogous statements for the spatially homogeneous Boltzmann equation. In the case of the classical Boltzmann equation, the nonlinearities are quadratic, unlike the cubic nonlinearities of the wave turbulence equation, and the equilibrium is the Maxwellian distribution. Depending on the collision kernel, it is known that the linearized operator around the equilibrium possesses a spectral gap (for instance in the hard sphere case), leading to exponential convergence to equilibrium.  However, in the soft potential case, there is no spectral gap, and the decay of the linearized operator is at most of order $e^{-t^{\sigma}}$, $\sigma <1$. For more details, see the reviews and articles \cite{Villanireview,Caflish80,GressmanStrain10} and references therein. 

Roughly speaking, the slightly weaker decay under soft potentials can be attributed to the fact that the collision frequency in the kernel of the Boltzmann operator is not lower-bounded and degenerates at large velocities. However, since the perturbative regime is around a Maxwellian distribution, the initial data roughly resembles a Maxwellian, making the large velocity regime, where the spectral gap becomes zero, less significant.

In our case, we also encounter a degenerate spectral gap that becomes zero at the edges of the frequency domain. This type of degeneracy is very different in nature and significantly affects the decay of the linearized semigroup.

Finally, the spatially inhomogeneous Boltzmann equation has been extensively studied, presenting additional mathematical challenges related to the existence of local equilibria. The theory of hypocoercivity has been developed for such inhomogeneous kinetic models, see \cite{ICMVillani2007, DesvVill05, StrainGuo08} and references therein.

\subsection*{Acknowledgements} Pierre Germain was supported by a Wolfson fellowship from the Royal Society and the Simons collaboration on Wave Turbulence. Joonhyun La acknowledges support from June Huh fellowship at Korea Institute for Advanced Study. Angeliki Menegaki acknowledges support from a Chapman fellowship at Imperial College London. 

\section{Notations}
For quantities $A$ and $B$ and parameters $a,b,c$, we write $A \lesssim_{a,b,c} B$ if there exists a constant $C(a,b,c)$ such that $A \leq C(a,b,c) B$.

We write $A \sim B$ if $A \lesssim B$ and $B \lesssim A$.

We write $C$ for a positive constant whose value may change from line to line.

If $b \in \mathbb{R}$, an exponent $b+$ is to be interpreted as follows: $f(t) \lesssim t^{b+}$ means that
$$
f(t) \lesssim_\epsilon t^{b+\epsilon} \qquad \mbox{for any $\epsilon>0$}.
$$

\section{Local well-posedness}
\label{sectionLWP}

\subsection{Local well-posedness in weighted $L^\infty$ spaces}

We consider here the Cauchy problem
$$
\begin{cases}
& \partial_t f = \mathcal{C}[f](p) \\
& f(t=0) = f_0,
\end{cases}
$$
where the collision operator is given in \eqref{formulacollision}.

\begin{theorem} \label{theo: Linfty lwp}
This Cauchy problem is locally well-posed in $\omega^{\alpha} L^\infty(\mathbb{T})$ if $\alpha \geq -1$. More precisely, if $f_0 \in \omega^\alpha L^\infty$, there exists a unique solution 
$f \in \mathcal{C}([0,T],\omega^\alpha L^\infty)$, where 
$$
T \gtrsim \| f_0 \|_{\omega^\alpha L^\infty}^{-2}.
$$
\end{theorem}

\begin{remark} What is the significance of the above theorem? The simplest case is $\alpha=0$, which simply gives local well-posedness in $L^\infty$. The case $\alpha =-1$ is important since $\omega^{-1} L^\infty$ is a function space which includes the singular Rayleigh-Jeans equilibria $\mathfrak{f}_{\beta,0}$. Finally, the case $\alpha \neq 0$ can be interpreted as the propagation by the nonlinear problem of the vanishing or singular behavior at $p=0 \mod 2\pi$ - related ideas will play an important role in the nonlinear stability questions examined later in this paper.
\end{remark}

\begin{proof} We claim that the collision operator is bounded on $\omega^\alpha L^\infty$. Taking this fact for granted for a moment, the equation can be written via Duhamel's formula as
$$
f = f_0 + \Phi[f] \quad \mbox{where} \quad \Phi[f] = \int_0^t \mathcal{C}[f](s) \dd s
$$
and the mapping $\Phi$ can be bounded in $\mathcal{C}([0,T],\omega^\alpha L^\infty)$ by
$$
\| \Phi[f] \|_{\mathcal{C}([0,T],\omega^\alpha L^\infty)} \lesssim T \| f \|_{\mathcal{C}([0,T],\omega^\alpha L^\infty)}^3. 
$$
Local well-posedness is then an easy consequence of the Banach fixed point theorem.

We now turn to proving boundedness of the collision operator in $\omega^\alpha L^\infty$. 
The denominator in the integrand in \eqref{formulacollision} is bounded by 
\begin{equation*}
\sqrt{ F_+(p_0,p_2)} = \sqrt{\left[ \cos \left( \frac{p_0}{2}\right) +\cos \left( \frac{p_2}{2}\right) \right]^2 +4\sin\left( \frac{p_0}{2}\right)\sin\left( \frac{p_2}{2}\right)} \geq 2 \sqrt{\omega_0 \omega_2}
\end{equation*}
implying that 
\begin{equation} 
\begin{split} 
|\mathcal{C}[f](t,p_0)| &\leq \int_0^{2\pi} \sqrt{\omega_0 \omega_2} \omega_1 \omega_3 \left| \prod_{\ell=0}^3 f_\ell \left(\frac{1}{f}+\frac{1}{f_1}-\frac{1}{f_2}-\frac{1}{f_3}\right) \right|
d p_2.
\end{split}
\end{equation} 

We now resort to the inequalities
\begin{align*}
& \sqrt{\omega_0 \omega_2} \omega_1 \omega_3  \omega^{-\alpha} \prod_{\ell=0}^3 \omega_\ell^{\alpha} \left({\omega^{-\alpha}}+ {\omega_1^{-\alpha}}+{\omega_2^{-\alpha}}+{\omega_3^{-\alpha}}\right)
\lesssim
\begin{cases}
\omega_2^{\frac{1}{2}-\alpha} & \mbox{if $-1 \leq \alpha \leq 0$} \\
1 & \mbox{if $\alpha \geq 0$}
\end{cases}
\end{align*}
(the second inequality being a consequence of $\omega_0 \lesssim \omega_1 \omega_2 \omega_3$, see Lemma \ref{lemmaomega123}).

As a consequence,
$$
\| \omega^{-\alpha} \mathcal{C}[f] \|_{ L^\infty} \lesssim \|\omega^{-\alpha} f \|_{L^\infty}^3 \int_0^{2\pi} \max(1,\omega_2^{\frac{1}{2} - \alpha}) \dd p_2 \lesssim \|\omega^{-\alpha} f \|_{L^\infty}^3.
$$
\end{proof}

\subsection{Unboundedness in weighted $L^p$ spaces for $p<3$}
In the preceding subsection, we saw that local well-posedness in weighted $L^\infty$ spaces was a consequence of the boundedness of the collision operator in these spaces. In the present section, we show that the collision operator cannot be bounded on $L^p$-type spaces for $p<3$. This does not imply ill-posedness, but this shows that well-posedness can only be the result of a very delicate nonlinear mechanism.

From now on we will write
$$
\underline{h}(x,z) = x-z + h(x,z) =  \frac{x-z}{2} + 2 \arcsin \left( \tan \frac{|z-x|}{4} \cos \frac{z+x}{4} \right)  \mod 2\pi.
$$

\begin{theorem} \label{thmunbounded}
The collision operator $\mathcal{C}$ is not bounded on $L^p(\mathbb{T})$ if $p<3$.
\end{theorem}

\begin{proof} \underline{Three specific points.} Given $x_0$, we choose $z_0$ such that $z \mapsto h(x_0 ,z)$ reaches its global maximum at $z_0$ (this maximum is unique modulo $2\pi$ for almost all values of $x_0$). It follows from this choice that $\underline{h}(x_0,z_0) = z_0$ (modulo $2\pi$). Indeed, there holds by symmetry between $p_2$ and $p_3$: $h(x_0,z_0) = h(x_0, \underline{h}(x_0,z_0))$; since $z \mapsto h(x_0 ,z)$ reaches its global maximum at $z_0$, this implies that $\underline{h}(x_0,z_0) = z_0$.

We set then
$$
p_0^0 = x_0, \quad p_1^0 = h(x_0,z_0), \quad p_2^0 = z_0, \quad p_3^0 = \underline{h}(x_0,z_0) = p_2^0.
$$

We claim that we can choose $x_0$ such that
\begin{enumerate}
\item The numbers $p^0_1,p^0_1, p^0_2$ are distinct and different from $0$ (modulo $2\pi$).
\item If the $(p_i)$ belong to the resonant manifold, if furthermore $p_0 = p_0^0$ and if finally $(p_2,p_3) = (p^0_{i_2},p^0_{i_3})$ for some $\{i_2,i_3\} \in \{0,1,2 \}$, then $(p_1,p_2,p_3) = (p_1^0,p_2^0,p_3^0)$.
\end{enumerate}

It is easiest to give a concrete example and to choose $x_0 =2$ and $z_0$ to be the local maximum of $z \mapsto h(x_0,z)$, or in other words
$$
p_0^0 = 2, \quad p_1^0 = 1.184..., \quad p_2^0 = 4.733...
$$
We now fix $p_0 = p_0^0 = 2$, and try all combinations of $i_2,i_3$ as in property (2) above
\begin{itemize}
\item If $p_2 = p^0_0$, then $p_3 = \underline{h}(p_0^0,p_0^0) = 0$.
\item If $p_2 = p^0_1$, then $p_3 = \underline{h}(p_0^0,p_0^1) = 0.698...$.
\item If $p_2 = p_2^0$, then $p_1= h(p_0^0,p_2^0) = p_1^0$ and $p_3 = \underline{h}(p_0^0,p_2^0) = p-_2$ modulo $2\pi$.
\end{itemize}

\medskip 
\noindent
\underline{The example giving unboundedness of the collision operator.}
We will test the collision operator on the following family of functions, which are uniformly bounded in $L^p$:
$$
f^\epsilon = \epsilon^{-\frac{2}{p}} \mathbf{1}_{[p_0^0,p_0^0+\epsilon^2]} + \epsilon^{-\frac{2}{p}} \mathbf{1}_{[p^0_1,p^0_1+\epsilon^2]} + \epsilon^{-\frac{1}{p}} \mathbf{1}_{[p^0_2,p^0_2+\epsilon]}.
$$

We will denote $K$ the integration kernel in the collision operator; it only vanishes if $p_i = 0$ for some $i=1,2,3,4$. The collision operator can be naturally split into a positive and a negative part
\begin{align*}
\mathcal{C}[f](p_0) & =  \int_0^{2\pi} K(p_0,p_2)\prod_{\ell=0}^3 f_\ell^\epsilon \left(\frac{1}{f^\epsilon}+\frac{1}{f_1^\epsilon}-\frac{1}{f_2^\epsilon}-\frac{1}{f_3^\epsilon}\right) \,dp_2 \\
& = \mathcal{C}_+[f](p_0) - \mathcal{C}_-[f](p_0),
\end{align*}
where it is understood that the variables $(p_0,p_1,p_2,p_3)$ in the integral are parameterized as $(p_0,h(p_0,p_2),p_2,\underline{h}(p_0,p_2))$.

On the one hand, if $|p_0-p_0^0| \ll \epsilon^2$, it follows from property (2) above that
\begin{align*}
\mathcal{C}_+[f^\epsilon](p_0) & = \int_0^{2\pi} K(p_0,p_2) \left[  f^\epsilon_1 f^\epsilon_2 f^\epsilon_3 + f^\epsilon_0 f^\epsilon_2 f^\epsilon_3  \right] \,dp_2 \\
& \lesssim \epsilon^{-\frac 4p} \int_0^{2\pi}  \left[  \mathbf{1}_{[p_1^0,p_1^0+\epsilon^2]} (p_1) \mathbf{1}_{[p_2^0,p_2^0+\epsilon]}(p_2) \mathbf{1}_{[p_2^0,p_2^0+\epsilon]}(p_3) \right. \\
& \qquad\qquad \qquad \qquad \left. +  \mathbf{1}_{[p_0^0,p_0^0+\epsilon^2]}(p_0)  \mathbf{1}_{[p_2^0,p_2^0+\epsilon]}(p_2) \mathbf{1}_{[p_2^0,p_2^0+\epsilon]}(p_3)  \right] \,dp_2 \\
& \lesssim \epsilon^{1 - \frac 4p}.
\end{align*}

On the other hand, still assuming $|p_0-p_0^0| \ll \epsilon^2$,
\begin{align*}
\mathcal{C}_-[f^\epsilon](p_0) & \geq \int_0^{2\pi} K(p_0,p_2)  f^\epsilon_0 f^\epsilon_1 f^\epsilon_2 \,dp_2 \\
& = \epsilon^{-\frac 5 p} \mathbf{1}_{[p_0^0,p_0^0+\epsilon^2]} (p_0)  \int_0^{2\pi} K(p_0,p_2) \mathbf{1}_{[p^0_1,p^0_1+\epsilon^2]} (p_1) \mathbf{1}_{[p^0_2,p^0_2+\epsilon]} (p_2) \,dp_2 \\
& \geq \epsilon^{1 - \frac 5p}.
\end{align*}

Overall, we get that that, if $|p_0 - p_0^0| \ll \epsilon^2$,
$$
\left| \mathcal{C}[f^\epsilon](p_0) \right| \gtrsim \epsilon^{1 - \frac 5p}.
$$
This implies that
$$
\| \mathcal{C}[f^\epsilon] \|_{L^p} \gtrsim \epsilon^{1 - \frac 3p},
$$
from which the desired result follows since $\| \mathcal{C}[f^\epsilon] \|_{L^p} \to \infty$ as $\epsilon \to 0$ if $p<3$.
\end{proof}

\section{First properties of the linearized operator}

\label{sectionfirstproperties}

Recall that the Rayleigh-Jeans equilibria are given by
$$ \mathfrak{f}(k) = \mathfrak{f}_{\beta, \gamma} (k) = \frac{1}{\beta \omega(k) + \gamma}, \qquad \beta,\gamma >0.$$

We now linearise around $\mathfrak{f}(k)$: If $f(t,k) = \mathfrak{f}(k) (1+ g(t,k))$, then $g(t,k)$ satisfies the following equation
\begin{equation} 
\begin{split} 
& \partial_tg(t,p)  = \mathfrak{f}(p)^{-1} \int_{\mathbb{T}^3} 
 \delta(\Sigma) \delta(\Omega)
  \left[ \prod_{\ell=0}^3 \omega_\ell \mathfrak{f}(p_\ell) \right]\times \\
  & \prod_{\ell=0}^3 (1+ g(t,p_\ell) ) \left(\frac{1- g(t,p)}{\mathfrak{f}(p)}+\frac{1- g(t,p_1)}{\mathfrak{f}(p_1)}-\frac{1- g(t,p_2)}{\mathfrak{f}(p_2)}-\frac{1- g(t,p_3)}{\mathfrak{f}(p_3)}\right)\,\dd p_{1,2,3} + \mathcal{O}(g^2)\\ 
& = [L g](k) + \mathcal{O}(g^2). 
 \end{split}
\end{equation} 
where
$$
[L g](p) = 
 \frac{1}{\mathfrak{f}(p)} \int_{\mathbb{T}^3} 
 \delta(\Sigma) \delta(\Omega)
  \left[ \prod_{\ell=0}^3 \omega_\ell \mathfrak{f}_\ell \right]
  \left[ -\frac{g}{\mathfrak{f}}-\frac{g_1}{\mathfrak{f}_1}+\frac{g_2}{\mathfrak{f}_2}+\frac{g_3}{\mathfrak{f}_3}\right]\,\dd p_{1,2,3}.
$$

The operator $L$ is symmetric in $L^2$ and the Dirichlet form is positive, namely
$$ \langle -L g,g \rangle = \frac{1}{4} \int \delta({\Sigma})\delta({\Omega}) \left[ \prod_{\ell=0}^3 \omega_\ell \mathfrak{f}_\ell \right] \left[ \frac{g_3}{\mathfrak{f}_3}+ \frac{g_2}{\mathfrak{f}_2} - \frac{g}{\mathfrak{f}}  - \frac{g_1}{\mathfrak{f}_1}\right]^2 \dd p_{1,2,3} \geq 0 .$$
 We split the linearised operator $L$ into $L = -A+ K$, where $A$ is the multiplication operator
$$
[Ag](p)= a(p) g(p), \qquad a(p) = \frac{\omega(p)}{\mathfrak{f}(p)}\int_{I^3}  \delta(\Sigma) \delta(\Omega)
\left[ \prod_{\ell=1}^3 \omega_\ell \mathfrak{f}_\ell \right]
\dd p_{1,2,3}
$$
and $K$ is the integral operator
$$
K = - K_1 + 2 K_2, \qquad 
\begin{cases}
& \displaystyle [K_1 g](p) = \frac{1}{\mathfrak{f}(p)} \int \delta(\Sigma) \delta(\Omega) \left[ \prod_{\ell=0}^3 \omega_\ell \mathfrak{f}_\ell \right] \frac{g_1}{\mathfrak{f}_1} \, \dd p_{1,2,3} \\
& \displaystyle [K_2 g](p) = \frac{1}{\mathfrak{f}(p)} \int \delta(\Sigma) \delta(\Omega) \left[ \prod_{\ell=0}^3 \omega_\ell \mathfrak{f}_\ell \right] \frac{g_2}{\mathfrak{f}_2} \, \dd p_{1,2,3}.
\end{cases}
$$
One checks that $K_1$ and $K_2$ are symmetric operators.

With the help of the technical lemmas \ref{lemmaintegrationp1} and \ref{lemmaintegrationp2}, the multiplier $a$ and the kernels $K_1$ and $K_2$ can be written as (we always abuse notations and identify an operator and its kernel)
\begin{equation}
\label{chardonneret}
\begin{split}
& a(p) = \frac{\omega(p)}{\mathfrak{f}(p)} 
\int_0^{2\pi}
\omega_1 \omega_2 \omega_3 \mathfrak{f}_1 \mathfrak{f}_2 \mathfrak{f}_3 \frac{\dd p_2}{\sqrt{ F_+(p,p_2)}}, \qquad p_1 = h(p,p_2),\; p_3 = p + p_1 - p_2 \\
& K_1(p,p_1) = \frac{\mathbf{1}_{F_-(p,p_1)>0}}{\sqrt{F_-(p,p_1)}}\omega  \omega_1 \omega_2 \omega_3 \mathfrak{f}_2 \mathfrak{f}_3, \qquad p_1 = h(p,p_2),\; p_3 = p + p_1 - p_2 \\
& K_2(p,p_2) = \frac{1}{\sqrt{F_+(p,p_2)}} \omega  \omega_1 \omega_2 \omega_3 \mathfrak{f}_1 \mathfrak{f}_3, \qquad p_1 = h(p,p_2),\; p_3 = p + p_1 - p_2.
\end{split}
\end{equation}

The next four lemmas lay out the basic properties of the linear operators $L,A,K_1,K_2$. These results already appear in \cite{LukkarinenSpohn2008}, except for the case $\gamma \neq 0$ of Lemma \ref{asymptoticsmultiplier}.

\begin{lemma}[Kernel of the linearized operator]
The kernel of $L$ is spanned by $\mathfrak{f}$ and $\omega \mathfrak{f}$.
\end{lemma}

\begin{proof}
This follows immediately from the characterization of collisional invariants, Theorem 2.2 in \cite{LukkarinenSpohn2008}.
\end{proof}

\begin{lemma}[Asymptotics of the multiplier function]
\label{asymptoticsmultiplier}
For any $\beta,\gamma>0$,
$$
a(p) \sim_{\beta,\gamma} \sin \left( \frac p 2 \right)^{\frac 53}.
$$
\end{lemma}

\begin{proof} This result is proved in lemmas \ref{lemmaasymptoticszero} and \ref{lemmaasymptoticsnonzero}, for zero and non-zero mass respectively. \end{proof}

\begin{lemma}[Compactness of the weighted $K_2$ operator]
For any $\beta$, $\gamma$, the operator $a^{-\frac 12} K_2 a^{-\frac 12}$ is compact on $L^2$. 
\end{lemma}

\begin{proof}
Using Lemma \ref{asymptoticsmultiplier} and the fact that $\sqrt{F_+} \geq \sqrt{\omega_0 \omega_2}$, we find that
$$
a(p)^{-\frac 12} K_2(p,p_2) a(p_2)^{-\frac 12} \lesssim \sin \left( \frac{p}{2} \right)^{-\frac 13}  \sin \left( \frac{p_2}{2} \right)^{- \frac 13}, 
$$
so that the corresponding operator is Hilbert-Schmidt, and therefore compact on $L^2$.
\end{proof}

\begin{lemma}[Compactness of the weighted $K_1$ operator]
For any $\beta$, $\gamma$, the operator $a^{-\frac 12} K_1 a^{- \frac 12}$ is compact on $L^2$.
\end{lemma}

\begin{proof}
Let $\widetilde{K}_1 = a^{-\frac 12} K_1 a^{- \frac 12}$. We will decompose this operator into
\begin{align*}
& \widetilde{K}_1(p,p_1)  = \widetilde{K}_1^{(1)}(p,p_1) + \widetilde{K}_1^{(2)}(p,p_1) + \widetilde{K}_1^{(3)}(p,p_1),\\
& \widetilde{K}_1^{(1)}(p,p_1) = \mathbb{1}_{\min(\omega,\omega_1) < \delta} \widetilde{K}_1(p,p_1) \\
& \widetilde{K}_1^{(2)}(p,p_1) = \mathbb{1}_{\substack{\min(\omega,\omega_1) > \delta \\ F_-(p,p_1) < \epsilon}} \widetilde{K}_1(p,p_1),
\end{align*}
where $\epsilon, \delta>0$ may vary. We claim that
$$
\| \widetilde{K}_1^{(1)} \|_{L^2 \to L^2} \lesssim \delta^{\frac 16}, \quad \| \widetilde{K}_1^{(2)} \|_{L^2 \to L^2} \lesssim \left( \frac \epsilon \delta \right)^{\frac 16}  \quad \mbox{and} \quad \| \widetilde{K}_1^{(3)}(p,p_1) \|_{L^\infty_{p,p_1}} \lesssim_{\epsilon,\delta} 1.
$$
These bounds imply that $\widetilde{K}_1^{(3)}$ is compact for any $\epsilon,\delta$, and that $\widetilde{K}_1$ can be approximated by compact operators in the operator norm. By closedness of the class of compact operators, this implies that $\widetilde{K}_1$ is compact.

There remains to prove the bounds on the operator norms of $\widetilde{K}_1^{(1)}$ and $\widetilde{K}_1^{(2)}$. We will rely on the following variant of Schur's test: given a weight $w$ and a symmetric operator $S$ with kernel $S(x,y)$,
$$
\| S \|_{L^2 \to L^2} \lesssim \sup_x w(x) \int |S(x,y)| \frac{\dd y}{w(y)}.
$$
Before applying this lemma, we record the following estimate for $\widetilde{K}_1$, which follows from Lemma \ref{lemmazerosF-}:
$$
\mbox{if $p_1 \in (0, p_1'(p))$}, \qquad |\widetilde{K}_1|(p,p_1) \lesssim (p_1'-p_1)^{-\frac 12} \omega^{-\frac 13} \omega_1^{\frac 16}.
$$
There is a symmetric estimate for $p_1 \in (p_1'',2\pi)$, but in order to make notations lighter, we will restrict the value of $p_1$ to $p_1 \in (0, p_1'(p))$. Then the variant of Schur's test above with the weight $w = \omega^{\frac 12}$ gives
\begin{align*}
\| \widetilde{K}_1^{(1)} \|_{L^2 \to L^2}  &\lesssim \sup_p
\int_0^{p_1'} (p_1'-p_1)^{-\frac 12} \omega^{\frac 16} \omega_1^{-\frac 13} \mathbb{1}_{\min(\omega,\omega_1) < \delta} \,\dd p_1 \\
& \lesssim \sup_{\omega < \delta} \int_0^{p_1'} (p_1'-p_1)^{-\frac 12} \omega^{\frac 16} \omega_1^{-\frac 13} \,\dd p_1 + \sup_p \int_0^{\min(\delta,p_1')}  (p_1'-p_1)^{-\frac 12} \omega^{\frac 16} \omega_1^{-\frac 13} \,\dd p_1 \lesssim \delta^{\frac 16},
\end{align*}
where the last inequality is a consequence of H\"older's inequality. 

Turning to $\widetilde{K}_1^{(2)}$, there holds on the support of its kernel $|p_1 - p_1'| \lesssim \frac{\epsilon}{\omega} \lesssim \frac{\epsilon}{\delta}$. Therefore,
\begin{align*}
\| \widetilde{K}_1^{(2)} \|_{L^2 \to L^2}  &\lesssim \sup_p
\int_{\max(0,p_1'-C \frac \epsilon \delta)}^{p_1'} (p_1'-p_1)^{-\frac 12} \omega^{\frac 16} \omega_1^{-\frac 13}  \,\dd p_1 \lesssim \left( \frac \epsilon \delta \right)^{\frac 16},
\end{align*}
which concludes the proof.
\end{proof}

Combining these four lemmas, we obtain a crucial lower bound on the linearized operator.

\begin{proposition} \label{propositiondissipation} For any $\beta,\gamma >0$,
$$ \text{for all } g \in L^2 \cap  \operatorname{Ker}(L)^\perp, \qquad
 \langle -L g,g\rangle \gtrsim \int a(p) |g(p)|^2 \dd p. $$
\end{proposition} 

\begin{proof}
Let $\widetilde{K} = a^{-\frac 12} K a^{- \frac 12}$. Then
$$
0 \leq - \langle L g,g \rangle = \langle A g,g \rangle - \langle K g,g \rangle = \langle A g,g \rangle - \langle \widetilde{K} a^{\frac 12} g, a^{\frac 12} g \rangle.
$$
Since $\widetilde{K}$ is compact and self-adjoint, its spectrum is discrete away from $0$. It follows from the above formula that $\widetilde{K}$ cannot have eigenvalues $>1$ and that the eigenspace associated to the eigenvalue $1$ coincides with the kernel of $L$. Finally, if $f \in \operatorname{Ker}(L)^\perp$, then $\| \widetilde{K} f \|_{L^2} \leq (1-\delta) \| f \|_{L^2}$ for some $\delta>0$, and thus
$$
 - \langle L g,g \rangle = \langle A g,g \rangle - \langle \widetilde{K} a^{\frac 12} g, a^{\frac 12} g \rangle \geq \int a(p) |g(p)|^2 \dd p.
$$
\end{proof}

Turning to the evolution problem, we note first that it admits a unique solution for $L^2$ data.

\begin{lemma}[Existence and uniqueness] For $g_0 \in L^2$, there exists a unique solution $g(t,p) \in \mathcal{C}([0,\infty],L^2)$ to the Cauchy problem 
$$
\begin{cases}
\partial_t g - Lg = 0 \\
g(t=0) = g_0
\end{cases}
$$
which we will denote $g(t) = e^{tL} g_0$. Furthermore, its $L^2$ norm is bounded by that of the data:
$$
\mbox{for any $t \geq 0$}, \qquad \| g(t) \|_{L^2} \lesssim \| g_0 \|_{L^2}.
$$
\end{lemma}

\begin{proof}
We saw that $a^{-\frac 12} K_1 a^{-\frac 12}$ and $a^{-\frac 12} K_2 a^{-\frac 12}$ are bounded on $L^2$, hence this is also the case for $K_1$ and $K_2$. Therefore, $L$ is bounded and the lemma follows by a fixed point theorem.
\end{proof}

As a consequence of the lower bound proved in the previous proposition, we get the following corollary.

\begin{corr}[Dissipation inequality]
\label{corodissi}
For $g_0 \in L^2$,
$$
\int_0^\infty \int a(p) |e^{tL} g_0(p)|^2 \dd p \dd t \lesssim \| g_0 \|_{L^2}^2.
$$
\end{corr}

This corollary quantifies time decay for the solution, but it will be insufficient for our purposes. Indeed, the equation is ill-posed in weighted $L^2$ spaces, so that this topology cannot be used to control nonlinear terms.

In the next section, we aim at obtaining pointwise decay, which will correct this shortcoming.

\section{Pointwise decay for the linearized operator}

\label{sectionpointwise}

In all that follows we assume that the chemical potential $\gamma>0$, that is we study the linearised operator around non-singular RJ equilibria. 

In this section we investigate how the energy dissipation leads to a polynomially fast relaxation for the linearized semigroup, pointwise and away from the edges. To this purpose we explore how the edges of the domain, where the weight in the Poincaré Inequality of the previous section vanishes, interact with the bulk of the domain, where the weight is lower-bounded. 

\subsection{Energy decay in the bulk}

We define
$$
\langle t \rangle = 10 + |t|
$$
and the following subintervals of $[0,2\pi]$, corresponding to the edges and the bulk of the domain respectively
\begin{equation}
\begin{split}
 & \mathcal{E}_{t,\alpha} = \{ p \in [0,2\pi]: p< \langle t \rangle^{-\alpha},\ p> 2\pi-\langle t \rangle^{-\alpha}  \} \text{ and}\\ 
 &  \mathcal{B}_{t,\alpha} = \{ p \in [0,2\pi]: \langle t \rangle^{-\alpha} \leq p \leq 2\pi- \langle t \rangle^{-\alpha}\}.
\end{split}
\end{equation}

We now define the following functionals: 
\begin{equation}
    \begin{split}
    m(t) = 
    \int_{\mathcal{B}_{t,\alpha}} |g(t,p)|^2 d p, \quad
    n(t)  = \int_{\mathcal{E}_{t,\alpha}  }|g(t,p)|^2 d p, \quad
    q(t) = \sup_{p\ \in\ \mathcal{E}_{t,\alpha} } |g(t,p)|.
 \end{split}
\end{equation}

\begin{lemma} \label{lemm:ODEf for m}
Assume that $\alpha < \frac 35$, and that 
$$
m(0)+n(0) \leq 1 \qquad \mbox{and} \qquad q(t) \leq \langle t \rangle^{e}
$$
where $e \in \mathbb{R}$. Then
$$
m(t) \lesssim_\alpha \langle t \rangle^{2e - \alpha}.
$$
\end{lemma}

\begin{proof}
By Lemma \ref{asymptoticsmultiplier} and Proposition \ref{propositiondissipation},
\begin{equation} \label{theODE}
\begin{split}
-\frac{d}{dt} (m(t) + n(t)) & = - 2\int g(t,p) L g(t,p) \dd p  \gtrsim \int a(p) |g(t,p)|^2 \dd p \\ 
&\gtrsim \int_{\mathcal{B}_{t,\alpha}} \omega(p)^{\frac 53} |g(t,p)|^2  \dd p \gtrsim \langle t \rangle^{- \frac{5\alpha}{3}} m(t).
    \end{split}
\end{equation}
We also observe that 
 $$ n(t) \lesssim q(t)^2 \langle t \rangle^{-\alpha} \lesssim \langle t \rangle^{2e-\alpha}.$$

Integrating the differential inequality \eqref{theODE} gives 
$$
m(t) \leq m(0) e^{-CI(t)} - e^{-CI(t)} \int_0^t n'(s) e^{CI(s)} \dd s, \qquad \mbox{where} \qquad I(t) = \int_0^t \langle s \rangle^{-\frac {5\alpha} 3} \dd s \sim \langle t \rangle^{1 - \frac {5\alpha}{3}},
$$
which becomes after integrating by parts in $s$
$$
m(t) + n(t) \leq (m(0) + n(0)) e^{-CI(t)} + C e^{-CI(t)} \int_0^t n(s) I'(s) e^{CI(s)} \dd s.
$$

The first term in the right-hand side is decaying exponentially fast, so we turn immediately to the second term. Using the bound on $n$ above,
$$
e^{-CI(t)} \int_0^t n(s) I'(s) e^{CI(s)} \dd s
\lesssim e^{- C \langle t \rangle^{1 - \frac{5\alpha} 3}} \int_0^t \langle s \rangle^{2e - \frac{8 \alpha}{3}} e^{C \langle s \rangle^{1 - \frac{5\alpha} 3}} \dd s.
$$

We now resort to the identity
\begin{equation} \label{identity: integral order}
\int_1^T e^{t^a} t^b \dd t \sim \int_1^{T^a} e^s s^{\frac b a + \frac 1 a -1} \dd s \sim e^{T^a} T^{b-a+1}, \qquad \mbox{if $T>2$, $a>0$, $b \in \mathbb{R}$}
\end{equation}
(which is itself a consequence of $\int_1^t e^s s^a \dd s \sim e^t t^a$, valid for any $a \in \mathbb{R}$) to see that
$$
e^{-I(t)} \int_0^t n(s) I'(s) e^{I(s)} \dd s 
\lesssim \langle t \rangle^{2e - \alpha}.
$$
\end{proof}

\subsection{A weak pointwise bound}
We prove a very weak bound, which will be a stepping stone to start the proof of pointwise decay.

\begin{lemma} \label{lem:weakLinfty}
Let $g$ solve $\partial_t g - L g =0$ with $g(t=0) =g_0 \in L^\infty$. Then
$$
\| g(t) \|_{L^\infty(I)} \lesssim \langle t \rangle^{1+} \| g_0 \|_{L^\infty(I)}. 
$$
 \end{lemma}

 \begin{proof}
By the decomposition of $L$ in Section \ref{sectionfirstproperties}, $f$ solves the equation
$$
\partial_t g = -ag - K_1 g + 2K_2 g. 
$$
We now want to bound the right-hand side in order to obtain an ODE satisfied by $\| f(t) \|_{L^\infty}$. Since $a \geq 0$, 
$$
\frac{d}{dt} \| g \|_{L^\infty} \lesssim \| K_1 g \|_{L^\infty} + \| K_2 g \|_{L^\infty}.
$$
The kernel of $K_2$ is $\displaystyle \frac{\omega \omega_1 \omega_2 \omega_3 \mathfrak{f}_1 \mathfrak{f}_3}{F_+(p,p_2)} \lesssim \sqrt{\omega \omega_2} \omega_1 \omega_3$; hence it is uniformly bounded and
$$
\| K_2 g(t) \|_{L^\infty} \lesssim \| g(t) \|_{L^2} \lesssim \|g_0\|_{L^\infty},
$$
since the $L^2$ norm of $f$ is decreasing.

Turning to $K_1$, its kernel is $\displaystyle \frac{\omega \omega_1 \omega_2 \omega_3 \mathfrak{f}_2 \mathfrak{f}_3}{\sqrt{F_-(p,p_1)}} \mathbf{1}_{F_-(p,p_1)>0}$.
We learn from Lemma \ref{lemmazerosF-} that $F_-$ has two nontrivial zeros, $p_1'$ and $p_1''$ and that $F_-(p,\cdot)$ is only positive on  $[0,p_1'] \cup [p_1'',2\pi]$. We only consider the first interval for simplicity and we apply Lemma \ref{lemmaomega123} followed by lower bound on $F_-$ in Lemma \ref{lemmazerosF-} to get that
\begin{align*}
\left| \int_0^{p_1'} \frac{\omega \omega_1 \omega_2 \omega_3 \mathbf{1}_{F_->0}}{\sqrt{F_-}} g(t,p_1) \dd p_1 \right| & \lesssim \omega^2 \left| \int \frac{ \mathbf{1}_{F_->0}}{\sqrt{F_-}} g(t,p_1) \dd p_1 \right| \\
& \lesssim \omega^{\frac 32} \left| \int_0^{p_1'} \frac{1} {\sqrt{p_1'-p_1}} g(t,p_1) \dd p_1 \right|.
\end{align*}

We then split the integral between $|p-p_1'| < \langle t \rangle^{-N}$ and $|p-p_1'| < \langle t \rangle^{-N}$ for $N$ large enough and bound the above by
\begin{align*}
 \int_{|p_1-p_1'| > \langle t \rangle^{-N}} \left| \frac{g(t,p_1)}{\sqrt{p_1'-p_1}} \right| \dd p_1 + 
 \int_{|p_1-p_1'| < \langle t \rangle^{-N} } 
\left| \frac{g(t,p_1)}{\sqrt{p_1'-p_1}} \right| \dd p_1.
\end{align*}
Bounding the first integral by the Cauchy-Schwarz inequality and the second by the $L^\infty$ norm of $f$, this is less than
$$
 \log \langle t \rangle \| g (t)\|_{L^2} + \langle t \rangle^{-\frac N2} \| g(t) \|_{L^\infty} \lesssim  \log \langle t \rangle \| g_0\|_{L^2} + \langle t \rangle^{-\frac N2} \| g(t) \|_{L^\infty}
$$
since the $L^2$ norm of $f$ is decreasing.

Overall, we find the ODE
$$
\frac{d}{dt} \| g(t) \|_{L^\infty} \lesssim  \log \langle t \rangle \| g_0\|_{L^2} + \langle t \rangle^{- \frac N2} \| g(t) \|_{L^\infty},
$$
which gives the desired result upon integration.
\end{proof}

\subsection{Pointwise bounds}

 \begin{lemma} \label{lem: Step 1 of the iteration}Let $\alpha < \frac{3}{5}$, and assume that
$$
\| g_0 \|_{L^\infty} \leq 1 \quad \mbox{and} \quad \| g(t) \|_{L^2} \leq \langle t \rangle^b.
$$
\begin{itemize}
\item[(i)] If $b>-1$, for any $p \in [0,2\pi]$,
$$|g(t,p)| \lesssim |g_0(p)| + \omega^{\frac{3}{2}} \langle t \rangle^{b+1+}  . $$
\item[(ii)] In the bulk: if $\omega > \langle t \rangle^{-\alpha}$,
$$|g(t,p)| \lesssim \omega^{-\frac 16} \langle t \rangle^{b+}$$
\end{itemize}
 \end{lemma}
\begin{proof}
$(i)$ We first split as above into the two integral kernels
\begin{align*}
\frac{d}{dt} |g(p)| & \lesssim  \left| \int \frac{\omega \omega_1 \omega_2 \omega_3}{\sqrt{F_+}} g(p_2) \dd p_2 \right|  + \left| \int \frac{\omega \omega_1 \omega_2 \omega_3 1_{F_- >0}}{\sqrt{F_-}} g(p_1) \dd p_1 \right| = I_+ + I_-.
\end{align*}
For $I_+$, we use successively the inequality $\sqrt{F_+} \gtrsim \sqrt{\omega \omega_2}$ and the Cauchy-Schwarz inequality to obtain
\begin{align*}
I_+& \lesssim \left| \int \sqrt{\omega \omega_2} \omega_1 \omega_3 g(p_1) \dd p_1 \right|   
 \lesssim \omega^{\frac 12} \left[ \int \omega_2 \omega_1^2 \omega_3^2 \,\dd p_1 \right]^{\frac 12} \| g \|_{L^2} \lesssim \omega^{\frac 32} \langle t \rangle^b.
\end{align*}

For the singular term $I_-$, we first use Lemma \ref{lemmaomega123} to obtain that
$$
I_- \leq \left| \int \frac{\omega \omega_1 \omega_2 \omega_3 \mathbf{1}_{F_->0}}{\sqrt{F_-}} g(p_1) \dd p_1 \right| \lesssim \omega \left| \int \frac{\omega \mathbf{1}_{F_->0}}{\sqrt{F_-}} g(p_1) \dd p_1 \right|.
$$

Next, we resort to Lemma \ref{lemmazerosF-}, from which we learn that $F_-(p,\cdot)$ is positive in $[0,p_2'] \cup [p_2'',2\pi]$; for simplicity, we will focus on the left interval. Using the lower bound on $F_-$ in Lemma \ref{lemmazerosF-}, we obtain that
$$
\omega \left| \int_0^{p_2'} \frac{\omega \mathbf{1}_{F_->0}}{\sqrt{F_-}} g(p_2) \dd p_2 \right| \lesssim \omega^{\frac 32} \left| \int_0^{p_2'} \frac{1} {\sqrt{p_2'-p_2}} g(p_2) \dd p_2 \right| 
$$

Finally, splitting the integral between $|p-p_1'| < \langle t \rangle^{-N}$ and $|p-p_1'| < \langle t \rangle^{-N}$ for $N$ big enough, and using Lemma \ref{lem:weakLinfty}, we get
$$
I_- \lesssim \omega^{\frac 32} \left[ \log \langle t \rangle \| g \|_{L^2} + \langle t \rangle^{-\frac N 2} \| g \|_{L^\infty} \right] \lesssim \omega^{\frac 32} \langle t \rangle^{b+} +  \omega^{\frac 32} \langle t \rangle^{2-\frac N 2} .
$$

We now combine the estimates on $I_+$ and $I_-$ to integrate the previous ODE in time, which gives 
$$
|g(t,p)| \lesssim | g_{0}(p) |+ \omega^{\frac 32} \langle t \rangle^{ b + 1+}.
$$

\medskip

\noindent $(ii)$ We proceed as in $(i)$, except that we do not neglect the term $a(p) f(t,p)$. This gives the differential inequality
$$
\left| \frac{d}{dt} g(t,p) + a(p) g(t,p) \right| \lesssim \omega^{\frac 32} \langle t \rangle^{b+}
$$
or
$$
\left| \frac{d}{dt} \left[ e^{a(p) t} g(t,p) \right] \right| \lesssim e^{a(p) t} \omega^{\frac 32} \langle t \rangle^{b+},
$$
which can be integrated to give
$$
|g(t,p)| \lesssim e^{- a(p) t} g_0(p) + \omega^{\frac 32} e^{-a(p) t} \int_0^t e^{a(p)s} \langle s \rangle^{b+} \dd s.
$$
If $\omega > \langle t \rangle^{-\alpha}$, the first term on the right-hand side decays faster than any polynomial. Turning to the second term, we bound it in a straightforward way and use that $a(p) \sim \omega^{5/3}$ to get
$$
\left| \omega^{\frac 32} e^{-a(p) t} \int_0^t e^{a(p)s} \langle s \rangle^{b+} \dd s \right| \lesssim \omega^{\frac 32} \langle t \rangle^{b+} \int_0^t e^{a(p)s} \dd s \lesssim \omega^{- \frac 16} \langle t \rangle^{b+}.
$$
\end{proof}

\subsection{Iterative improvement}

We will now apply iteratively lemmas \ref{lemm:ODEf for m} and \ref{lem: Step 1 of the iteration}, having set $\alpha = \frac 35 - \epsilon_0$, with $0 <\epsilon_0 \ll 1$. In the following, we denote $t^{a +}$ for $t^{a + C\epsilon_0}$, for a constant $C$, instead of our usual notation $t^{a+}$ meaning $t^{a + \delta}$ for any $\delta>0$.

\begin{itemize}
\item Applying first Lemma \ref{lem: Step 1 of the iteration} $(i)$ with $b=0$ gives, provided $f_0 \in L^2$,
$$
|g(t,p)| \lesssim |g_0(p)| + \omega^{\frac 32} \langle t \rangle^{1+} .
$$
Thus, assuming $|f_0(p)| \leq 1$, we get
$$
q(t) \lesssim 1 + \langle t \rangle^{1 - \frac 32 \cdot \frac 35 +} \lesssim \langle t \rangle^{ \frac 1 {10} +}.
$$

\item Applying Lemma \ref{lemm:ODEf for m} with $e = \frac 1 {10} +$ gives
$$
m(t) \lesssim \langle t \rangle^{\frac 15 - \frac 35 +} = \langle t \rangle^{- \frac 25 +}.
$$
This gives
$$
\| g(t) \|_{L^2} \lesssim m(t)^{\frac 12} + q(t) \langle t \rangle^{-\frac 3 {10} +} \sim \langle t \rangle^{-\frac 15 +}.
$$

\item Applying Lemma \ref{lem: Step 1 of the iteration} $(i)$ with $b = - \frac 15 +$ gives
$$
|g(t,p)| \lesssim |g_0(p)| + \omega^{\frac 32} \langle t \rangle^{\frac 45+}.
$$
Assuming $|f_0(p)| \lesssim \omega^{\frac 16}$, this gives
$$
q(t) \lesssim \langle t \rangle^{- \frac 16 \cdot \frac 35 +} + \langle t \rangle^{\frac 45 - \frac 32 \cdot \frac 35 +} \sim \langle t \rangle^{- \frac 1 {10} +}.
$$

\item Applying Lemma \ref{lemm:ODEf for m} with $e = - \frac 1 {10}$ gives
$$
m(t) \lesssim \langle t \rangle^{-\frac 15 - \frac 35 +} = \langle t \rangle^{-\frac 45 +},
$$
hence
$$
\| g(t) \|_{L^2} \lesssim \langle t \rangle^{-\frac 25 +} + \langle t \rangle^{-\frac 1 {10} - \frac 3 {10} +} \sim \langle t \rangle^{-\frac 25 +}.
$$

\item Applying Lemma \ref{lem: Step 1 of the iteration} $(i)$ with $b = - \frac 25$, we get
$$
|g(t,p)| \lesssim |g_0(p)| + \omega^{\frac 32 } \langle t \rangle^{\frac 35+}.
$$
If $|f_0(p)| \leq \omega^\nu$ with $\nu \in [\frac 16,\frac 12]$, this gives
\begin{equation}
\label{boundq}
q(t) \lesssim \langle t \rangle^{- \frac 35 \nu +} + \langle t \rangle^{- \frac 3 {10} +} \sim  \langle t \rangle^{- \frac 35 \nu +}.
\end{equation}

\item Applying Lemma \ref{lemm:ODEf for m} with $e =- \frac 35 \nu +$ gives finally
$$
m(t) \lesssim \langle t \rangle^{- \frac {6 \nu}{5} - \frac 35 +}
$$
hence
\begin{equation}
\label{boundm}
\| g(t) \|_{L^2} \lesssim \langle t \rangle^{- \frac {3 \nu}{5} - \frac 3{10} +}. 
\end{equation}
\end{itemize}

We can now prove our final pointwise bound, under the assumption made above that $|f_0(p)| \leq \omega^\nu$. In the edges, we use \eqref{boundq} to get that
$$
\omega^\mu | g(t,p) | \lesssim \langle t \rangle^{-\frac 35 (\mu + \nu) +} \qquad \mbox{if $\omega < \langle t \rangle^{-\alpha}$}.
$$
In the bulk, we use Lemma \ref{lem: Step 1 of the iteration} $(ii)$ and \eqref{boundm} to get that
$$
\omega^{\frac 16} | g(t,p) | \lesssim \langle t \rangle^{- \frac {6 \nu + 3}{10}} \qquad \mbox{if $\omega > \langle t \rangle^{-\alpha}$}.
$$
Combining these two bounds results in the following theorem.

\begin{theorem}[Pointwise decay] \label{theoremdecay}
Assume that $g_0 \in L^\infty$, with a zero projection (in $L^2$) on $\operatorname{Ker} L$. Then for any $\mu ,\nu \in (\frac 16, \frac 12)$ and any $\delta>0$
$$
\| \omega^\mu e^{tL} g \|_{L^\infty} \lesssim_\delta \langle t \rangle^{- \frac{3}{5}(\mu + \nu)+ \delta} \| \omega^{-\nu} g_0(p) \|_{L^\infty}.
$$
\end{theorem}

\section{Nonlinear stability}
\label{sectionnonlinear}
\begin{theorem}
\label{thmnonlinstab}
For any $\beta,\gamma >0$, there exists $\epsilon_0$ such that the following holds. If
$$
f(t=0) = \mathfrak{f}_{\beta,\gamma} [ 1 + g_0]
$$
where 
$$
\int \mathfrak{f}_{\beta,\gamma} g_0 \dd p = \int \omega \mathfrak{f}_{\beta,\gamma} g_0 \dd p = 0
$$
and
$$
\| \omega^{-\frac 12} g_0 \|_{L^\infty} = \epsilon < \epsilon_0,
$$
then there exists a global solution which can be written
$$
f(t) = \mathfrak{f}_{\beta,\gamma} [ 1 + g] \qquad \mbox{where} \quad \| \omega^{\frac{1}{2}} g(t,p) \|_{L^\infty_p} \lesssim \epsilon \langle t \rangle^{-\frac 35 + \frac{1}{1000}} \quad \mbox{for all $t \geq 0$}.
$$
\end{theorem}

\begin{proof}
The full equation satisfied by the perturbation $g$ is
\begin{equation}
\label{equationg}
\partial_t g - L g = \mathcal{Q}(g) + \mathcal{C}(g),
\end{equation}
where
\begin{align*} 
& \mathcal{Q}[g](t,p_0)= \frac{2}{\mathfrak{f}_0} \int_0^{2\pi}
\frac{ \omega_0 \omega_1 \omega_2 \omega_3
 }{\sqrt{F_+(p_0,p_2)}}  \left[ \mathfrak{f}_2 \mathfrak{f}_3 g_2 g_3 - \mathfrak{f}_0 \mathfrak{f}_1 g_0 g_1 \right] \dd p_2 \\
&  \mathcal{C}[g](t,p_0)= \frac{1}{\mathfrak{f}_0} \int_0^{2\pi}
\frac{ \omega_0 \omega_1 \omega_2 \omega_3
 }{\sqrt{F_+(p_0,p_2)}}  \prod_{\ell=0}^3 \mathfrak{f}_\ell g_\ell \left(\frac{1}{\mathfrak{f} g}+\frac{1}{\mathfrak{f}_1 g_1}-\frac{1}{\mathfrak{f}_2 g_2}-\frac{1}{\mathfrak{f}_3 g_3}\right)
 \dd p_2.
\end{align*} 

Duhamel's formula gives the equivalent formulation
\begin{equation}
\label{Duhamelform}
g(t) = S_t g_0 + \int_0^t S_{t-s} \left( \mathcal{Q}[g](s,p) +  \mathcal{C}[g](s,p)\right) \dd s.
\end{equation}

The key norm that will be used to analyze this problem is, for $T>0$,
\begin{equation} \label{def:B_t norm}
\| g \|_{\mathcal{B}_T} := 
 \left\| \langle t \rangle^{\frac 2 5 - \delta} \omega^{\frac 1 6} g \right\|_{L^\infty_{t,p}([0,T] \times \mathbb{T})} + \left\| \langle t \rangle^{\frac 3 5 - \delta} \omega^{\frac 1 2} g \right\|_{L^\infty_{t,p}([0,T] \times \mathbb{T})}.
\end{equation}

Our first lemma gives local well-posedness in $ \omega^{- \frac 16} L^\infty$ for this equation.

\begin{lemma}[Local well-posedness in $\omega^{-\frac 16} L^\infty$] 
\label{LWPg}
The equation \eqref{equationg} admits a unique maximal solution $g \in \mathcal{C}([0,T_0), \omega^{-\frac 16} L^\infty)$, where $T_0 >0$ and
$$
\lim_{t \to T_0} \| \omega^{\frac 16} g(t,p) \|_{L^\infty_p} = \infty \qquad \mbox{if $T_0 < \infty$}
$$
Furthermore, there exists a constant $C_0$ such that
$$
\lim_{t \to 0} \| g \|_{\mathcal{B}_t} \leq C_0 \epsilon,  
$$
where $\epsilon := \| \omega^{-\frac 12} g_0 \|_{L^\infty}$ and where the norm $\|\cdot\|_{\mathcal{B}_T}$ is defined in \eqref{def:B_t norm}.
\end{lemma}

We postpone the proof of this lemma for the time being, and admit its statement. 
We aim at proving that $T_0 = \infty$, and that the solution $g$ is actually decaying. This will be achieved through a boostrap argument bearing on the quantity $\| f \|_{\mathcal{B}_t}$. This bootstrap argument will rely on the two following lemmas, whose proofs we postpone for the moment.

\begin{lemma}[A priori bound] 
\label{apriori}
There exists a constant $C_0 >0$ such that: for any $T>0$, if $g$ is a solution in $\mathcal{C}([0,T],\omega^{- \frac 16} L^\infty)$, then
$$
\| g \|_{\mathcal{B}_T} \leq C_0 \left[ \epsilon + \| g \|_{\mathcal{B}_T}^2 + \| g \|_{\mathcal{B}_T}^3 \right]
$$
where $\epsilon := \| \omega^{-\frac 12} g_0 \|_{L^\infty}$ and where $\|\cdot\|_{\mathcal{B}_T}$ is defined in \eqref{def:B_t norm}.
\end{lemma}

In the above lemma, we denoted $C_0$ for the constant, just like in Lemma \ref{LWPg}; it suffices to take $C_0$ to be the largest of the two to avoid any confusion.

\begin{lemma}[Bootstrap inequality] 
\label{bootstrap}
If $x(t)$ is a continuous function on $(0,T)$ such that
$$
x(t) \leq C_0 \left[ \epsilon + x(t)^2 + x(t)^3 \right] \qquad \mbox{and} \qquad 
\lim_{t\to 0} x(t) \leq C_0 \epsilon,
$$
then  
$$
x(t) \leq 2C_0 \epsilon \qquad \mbox{for any} \; t \in (0,T).
$$
\end{lemma}

We now perform the bootstrap argument: 
arguing by contradiction, let us assume that $T_0 < \infty$. By Lemma \ref{apriori} and Lemma \ref{bootstrap}, we get that $\| g \|_{\mathcal{B}_t} \leq 2 C_o \epsilon$ for all $t < T_0$. But this contradicts the blow up criterion in Lemma \ref{LWPg}. Therefore, $T_0 = \infty$. A new application of Lemma \ref{apriori} and Lemma \ref{bootstrap} gives the desired result!
\end{proof}

\begin{proof}[Proof of Lemma \ref{LWPg}] It is very similar to the proof of Theorem \ref{theo: Linfty lwp} above and Lemma \ref{apriori} below, and will therefore be omitted.
\end{proof}

\begin{proof}[Proof of Lemma \ref{apriori}] The linear term on the right-hand side of \eqref{Duhamelform} can be bounded by Theorem \ref{theoremdecay}:
$$
\left\| S_t g_0 \right\|_{\mathcal{B}_\infty} \lesssim \| \omega^{-\frac 12} g \|_{L^\infty} < \epsilon 
$$

To deal with the quadratic and cubic terms, we will use the conservation laws of mass and energy. They imply that 
\begin{align*}
& \int \mathfrak{f}_{\beta,\gamma}(p)(1 + g(t,p)) \dd p = \int \mathfrak{f}_{\beta,\gamma}(p) \dd p \\
& \int \omega(p) \mathfrak{f}_{\beta,\gamma}(p)(1 + g(t,p)) \dd p = \int \omega(p) \mathfrak{f}_{\beta,\gamma}(p) \dd p 
\end{align*}
or in other words
$$
\int \mathfrak{f}_{\beta,\gamma}(p) g(t,p) \dd p = \int \omega(p) \mathfrak{f}_{\beta,\gamma}(p) g(t,p) \dd p = 0.
$$

As a consequence, the orthogonal projection (in $L^2$) of $g$ and $Lg$ on $\operatorname{Ker} L$ is zero. For the equation \eqref{Duhamelform} satisfied by $g$, this implies that the projection of the quadratic and cubic terms on $\operatorname{Ker} L$ is also zero. Therefore, we can apply Theorem \ref{theoremdecay} with $\mu = \nu = \frac 12$: if $t \in [0,T]$,
\begin{align*}
& \left\| \omega^{\frac 12} \int_0^t S_{t-s} \left( \mathcal{Q}[g](s,p) +  \mathcal{C}[g](s,p)\right) \dd s\right\|_{L^\infty_p} \\
& \qquad \lesssim \int_0^t  \langle t-s \rangle^{-\frac 35 + \delta} \left\| \omega^{-\frac 12}  \mathcal{Q}[g](s,p)\right\|_{L^\infty} \dd s + \int_0^t  \langle t-s \rangle^{-\frac 35 + \delta} \left\| \omega^{-\frac 12}  \mathcal{C}[g](s,p) \right\|_{L^\infty} \dd s\\
& \qquad = I + II.
\end{align*}
Using that $F_+ \gtrsim \omega_0 \omega_2$ and $\omega_1 \omega_2 \omega_3 \lesssim \omega_0$,
\begin{align*}
I &\lesssim \int_0^t \langle t-s \rangle^{-\frac 35 + \delta}\left\| \int \sqrt{\omega_2} \omega_1 \omega_3 [ |g_2 g_3| + |g_0 g_1|] \dd p_2 \right\|_{L^\infty} \dd s  \\
&\lesssim  \int_0^t \langle t-s \rangle^{-\frac 35 + \delta}  \| \omega^{\frac{1}{2}}g \|_{L^\infty}^2 \dd s  \lesssim \| g\|_{\mathcal{B}}^2 \int_0^t \langle t-s \rangle^{-\frac 35 + \delta}  \langle s \rangle^{\left(- \frac 35 + \delta \right)\cdot 2} \dd s \\
& \lesssim \| g \|_{\mathcal{B}}^2 \langle t \rangle^{- \frac 35 + \delta}.
\end{align*}
Similarly,
\begin{align*}
II & \lesssim  \int_0^t \langle t-s \rangle^{-\frac 35 + \delta}\left\| \int \sqrt{\omega_2} \omega_1 \omega_3 [ |g_1 g_2 g_3| +  |g_0 g_2 g_3| + |g_0 g_1 g_2| + |g_0 g_1 g_3| ] \dd p_2 \right\|_{L^\infty} \dd s  \\
& \lesssim  \int_0^t \langle t-s \rangle^{-\frac 35 + \delta} \| \omega^{\frac{1}{6}} g \|_{L^\infty} \| \omega^{\frac{1}{3}} g \|_{L^\infty} \| \omega^{\frac{1}{2}} g \|_{L^\infty} \dd s \\
& \lesssim  \| g \|_{\mathcal{B}}^3 \int_0^t \langle t-s \rangle^{-\frac 35 + \delta} \langle s \rangle^{ -\frac 25 - \frac 12 - \frac 35 + 4\delta} \dd s \\
& \lesssim   \| g\|_{\mathcal{B}}^3  \langle t \rangle^{-\frac 35 + \delta}.
\end{align*}

Using now Theorem \ref{theoremdecay} with $\mu = \frac 16$ and $\nu = \frac 12$,
\begin{align*}
& \left\| \omega^{\frac 16} \int_0^t S_{t-s} \left( \mathcal{Q}[g](s,p) +  \mathcal{C}[g](s,p)\right) \dd s\right\|_{L^\infty_p} \\
& \qquad \lesssim \int_0^t  \langle t-s \rangle^{-\frac 25 + \delta} \left\| \omega^{-\frac 12}  \mathcal{Q}[g](s,p)\right\|_{L^\infty} \dd s + \int_0^t  \langle t-s \rangle^{-\frac 25 + \delta} \left\| \omega^{-\frac 12}  \mathcal{C}[g](s,p) \right\|_{L^\infty} \dd s\\
& \qquad = III + IV.
\end{align*}

Taking once again advantage the inequalities $F_+ \gtrsim \omega_0 \omega_2$ and $\omega_1 \omega_2 \omega_3 \lesssim \omega_0$, 
\begin{align*}
III & \lesssim  \int_0^t \langle t-s \rangle^{-\frac 25 + \delta}\left\| \int \sqrt{\omega_2} \omega_1 \omega_3 [ |g_2 g_3| + |g_0 g_1|] \dd p_2 \right\|_{L^\infty} \dd s \\
& \lesssim \int_0^t \langle t-s \rangle^{-\frac 25 + \delta} \| \omega^{\frac 12} g \|_{L^\infty_p} \dd s \\
& \lesssim \| g \|_{\mathcal{B}_t}^2 \int_0^t  \langle t-s \rangle^{-\frac 25 + \delta} \langle s \rangle^{\left(\frac 12 + \delta \right) \cdot 2} \\
& \lesssim \| g \|_{\mathcal{B}_t}^2 \langle t \rangle^{- \frac 25 + \delta}.
\end{align*}

Similarly,
\begin{align*}
IV & \lesssim  \int_0^t \langle t-s \rangle^{-\frac 25 + \delta}\left\| \int \sqrt{\omega_2} \omega_1 \omega_3 [ |g_1 g_2 g_3| +  |g_0 g_2 g_3| + |g_0 g_1 g_2| + |g_0 g_1 g_3| ] \dd p_2 \right\|_{L^\infty} \dd s  \\
& \lesssim  \int_0^t \langle t-s \rangle^{-\frac 25 + \delta} \| \omega^{\frac{1}{6}} g \|_{L^\infty} \| \omega^{\frac{1}{3}} g \|_{L^\infty} \| \omega^{\frac{1}{2}} g \|_{L^\infty} \dd s \\
& \lesssim  \| g \|_{\mathcal{B}}^3 \int_0^t \langle t-s \rangle^{-\frac 25 + \delta} \langle s \rangle^{ -\frac 25 - \frac 12 - \frac 35 + 4\delta} \dd s \\
& \lesssim   \| g \|_{\mathcal{B}}^3  \langle t \rangle^{-\frac 25 + \delta}.
\end{align*}
This gives the desired estimate.
\end{proof}

\begin{proof}[Proof of Lemma \ref{bootstrap}] Consider the function
$$
F: x \mapsto x - C_0 (\epsilon + x^2 + x^3).
$$
It is clear that $F$ is negative on $[0,C_0 \epsilon + \eta]$, for some $\eta>0$. It is also clear that
$$
F(2C_0 \epsilon) = C_0 \epsilon - 4 C_0^3 \epsilon^2 - 8 C_0^4 \epsilon^3 = C_0 \epsilon (1 - 4 C_0^2 \epsilon - 8 C_0^3 \epsilon^2) \geq C_0 \epsilon (1 - 4 C_0^2 \epsilon_0 - 8 C_0^3 \epsilon_0^2) \geq 0
$$
provided $\epsilon_0$ is chosen sufficiently small. The statement of the lemma follows by the intermediate value theorem.
\end{proof}

\section{Mass and Energy of RJ spectra}

\label{MERJ}

In this section, we investigate the following question: given a mass and an energy $(\mathcal{M}_0, \mathcal{E}_0)$, is there a Rayleigh-Jeans spectrum $\mathfrak{f}$ such that $\mathcal{M}(\mathfrak{f}) = \mathcal{M}_0, \mathcal{E}(\mathfrak{f}) = \mathcal{E}_0$? It turns out that the ratio $\mathcal{E}_0/\mathcal{M}_0$ decides whether such a RJ spectrum exists. We define the following functions:
\begin{equation*}
\begin{split}
\mathcal{M} (b, g) &= \int_{\mathbb{T}} \frac{\dd p}{b^{-1} \omega + g^{-1}}, \\
\mathcal{E} (b,g) &=\int_{\mathbb{T}} \frac{\omega \dd p}{b^{-1} \omega + g^{-1}},
\end{split}
\end{equation*}
where $0<b, g < \infty.$ We readily see that $0< \mathcal{M}(b, g), \mathcal{E} (b,g) < \infty,$ and
\begin{equation*}
    b^{-1} \mathcal{E} + g^{-1} \mathcal{M} = 2\pi.
\end{equation*} 
Moreover, we have the following differential equation:
\begin{equation*}
    b \partial_b \mathcal{M} + g \partial_g \mathcal{M} = \mathcal{M}.
\end{equation*}
We adopt polar coordinates to represent $(b,g)$:
\begin{equation*}
    b = r \cos \theta, \qquad g = r \sin \theta,
\end{equation*}
where $0<r<\infty$ and $0 < \theta < \pi/2$. By homogeneity,
\begin{equation*}
    \mathcal{M} (b, g) = r \mathcal{M} (\cos \theta, \sin \theta).
\end{equation*}
Therefore, for a given pair of positive numbers $(\mathcal{M}_0, \mathcal{E}_0)$, there exists $(b, g)$ such that $(\mathcal{M}_0, \mathcal{E}_0) = (\mathcal{M}(b,g), \mathcal{E}(b,g))$ if and only if the following equation on $\theta$ is solvable:
\begin{equation*}
    \frac{\mathcal{E}_0}{\cos \theta} + \frac{\mathcal{M}_0}{\sin \theta} = 2\pi r = \frac{2\pi \mathcal{M}_0}{\mathcal{M}(\cos \theta, \sin \theta )},
\end{equation*}
or
\begin{equation*}
    \frac{\mathcal{E}_0}{\mathcal{M}_0} = \frac{2\pi \cos \theta}{\mathcal{M}(\cos \theta, \sin \theta)} - \cot \theta = \frac{1}{\frac{1}{2\pi} \int_{\mathbb{T}} \frac{\dd p}{\omega+\cot \theta}} - \cot \theta.
\end{equation*}
Since $\theta \rightarrow \cot \theta$ is a bijection from $(0, \pi/2)$ to $(0, \infty),$ this is equivalent to that 
\begin{equation*}
    \frac{\mathcal{E}_0}{\mathcal{M}_0} \in F((0, \infty)),
\end{equation*}
where $$F(\ell ) = \frac{1}{\frac{1}{2\pi} \int_{\mathbb{T} } \frac{\dd p}{\omega + \ell }} - \ell.$$
Note that 
$$
F(0)=0 \quad \mbox{while} \quad F(\infty)=\frac{2}{\pi}.
$$ 
If we let $G(\ell) = F(\ell) + \ell,$ we see that $$\frac{\dd}{\dd \ell} G(\ell) = \frac{\frac{1}{2\pi} \int_{\mathbb{T}} \frac{1}{(\omega+\ell)^2} \dd p }{\left ( \frac{1}{2\pi} \int_{\mathbb{T} } \frac{\dd p}{\omega + \ell }\right )^2 } > 1$$
by Jensen's inequality (since neither $x \rightarrow x^2$ is affine nor $p \rightarrow \frac{1}{\omega+\ell}$ is constant, the inequality is strict), which implies that $\ell \rightarrow F(\ell)$ is strictly increasing. To summarize, we have the following:
\begin{proposition}
    Let $(\mathcal{M}_0, \mathcal{E}_0)$ be a pair of positive numbers. Then there exists a Rayleigh-Jeans spectrum $\mathfrak{f} = \frac{1}{\beta \omega + \gamma}$, $0< \beta, \gamma < \infty$, such that $\mathcal{M} (\mathfrak{f} ) = \mathcal{M}_0$ and $\mathcal{E} (\mathfrak{f} ) = \mathcal{E}_0$ hold if and only if $$ 0 < \frac{\mathcal{E}_0}{\mathcal{M}_0} < \frac{2}{\pi}.$$ Moreover, if the latter is the case, such $\mathfrak{f}$ is unique and determined by the following:
    \begin{equation*}
    \begin{split}
        \beta &= \left ( r \cos \theta \right )^{-1}, \qquad \gamma = \left ( r \sin \theta \right )^{-1}, \\
        \theta &= \arctan \left [ \left (F^{-1} \left (\frac{\mathcal{E}_0}{\mathcal{M}_0}\right )\right )^{-1} \right ], \qquad r = \frac{ \mathcal{M}_0}{\frac{\cos\theta}{2\pi} \int_{\mathbb{T}} \frac{\dd p}{\omega + \cot \theta}}.
    \end{split}
    \end{equation*}
\end{proposition}

Before finishing this section, let us comment on a related conjecture made in \cite{Nazar23}. For the truncated Wave Kinetic Equation arising from NLS, the presence or absence of a finite-time blow-up is expected to depend on whether the ratio of mass to energy is sufficiently large. Thus, given our proposition above, we expect that initial data whose mass and energy do not correspond to a Rayleigh-Jeans distribution will form a condensate in finite time. Proving rigorously that these condensates are stationary solutions remains an interesting open problem.

\appendix

\section{Calculus lemmas}\label{App: calculus lemmas}

In this section, we gathered some nontrivial computations which are used in the rest of the article. Many of the formulas we derive already appeared in \cite{LukkarinenSpohn2008} and some are new. Many proofs are also close to \cite{LukkarinenSpohn2008} and are provided for completeness and for the reader's convenience.

\subsection{Integration on the resonant manifold}
If $x,y,z \in [0,2\pi]$, let 
$$
\Omega(x,y,z) = \sin\left( \frac x2 \right) + \sin \left( \frac y2 \right) - \sin \left( \frac z2 \right) - \left| \sin \left( \frac{x+y-z}{2} \right) \right|. 
$$

Recall that
$$
h(x,z) = \frac{z-x}{2} + 2 \arcsin \left( \tan \frac{|z-x|}{4} \cos \frac{z+x} 4 \right).
$$

\begin{lemma}[Parameterization of the resonant manifold]
\label{lemmaparameterization}
The zero set $Z$ of $\Omega$ on $[0,2\pi]^3$ can be split into $Z = Z_+ \cup Z_-$, where
\begin{align*}
& Z_+ = Z \cap \{ 0 \leq x + y - z \leq 2\pi \} \\
& Z_- = Z \cap \{ x+ y -z \leq 0 \; \mbox{or} \; x+ y - z \geq 2\pi \}.
\end{align*}

The set $Z_+$ consists of $(x,y,z) \in [0,2\pi]^3$ such that $x=y \in \{0,2\pi\}$ or $x=z$ or $y=z$.

As for $Z_-$, it can be described as follows: it consists of $(x,y,z) \in [0,2\pi]^3$ such that
\begin{itemize}
\item either $x<z$ and $y = h(x,z)$
\item or $x>z$ and $y = h(x,z) + 2\pi$.
\end{itemize}
\end{lemma}

\begin{proof} \underline{A basic inequality.} We start by proving that
\begin{equation}
\label{boundtan}
\left| \tan \left( \frac{z-x}{4} \right) \cos \left( \frac{x+z}{4} \right) \right| \leq 1 \quad \mbox{if $(x,z) \in [0,2\pi]^2$}.
\end{equation}
This inequality is a consequence of
$$
\left| \cos \left( \frac{x+z}{4} \right) \right| \leq \cos \left( \frac{x-z}{4} \right) \quad \mbox{if $(x,z) \in [0,2\pi]^2$}.
$$
To check the latter inequality, we observe that, on the one hand, $0 \leq \left|\frac{x-z}{4} \right| \leq \frac{x+z}{4} \leq \pi$, which implies that 
\begin{equation}
\label{star2}
\cos \left( \frac{x+z}{4} \right) \leq \cos \left( \frac{x-z}{4} \right).
\end{equation}
On the other hand, $0 \leq \left|\frac{x-z}{4} \right| \leq \frac{4 \pi - x - z}{4} \leq \pi$, which implies that 
\begin{equation}
    \label{star3}
\cos \left( \pi - \frac{x+z}{4} \right) = -\cos \left( \frac{x+z}{4} \right) \leq \cos\left( \frac{x-z}{4} \right).
\end{equation}

\medskip

\noindent 
\underline{Case 1: $0 \leq x + y - z \leq 2\pi $.} Then $\Omega$ can be written
$$
\Omega(x,y,z) = \sin\left( \frac x2 \right) + \sin \left( \frac y2 \right) - \sin \left( \frac z2 \right) - \sin \left( \frac{x+y-z}{2} \right)
$$
which becomes after using the trigonometric sum-to-product formulas
$$
\Omega(x,y,z) = 4 \sin \left( \frac{x+y}{4} \right)\sin \left( \frac{x-z}{4} \right)\sin \left( \frac{y-z}{4} \right).
$$
The description of $Z_+$ follows immediately from this formula.

\medskip

\noindent \underline{Case 2: $x+ y -z \leq 0$ or $x+ y - z \geq 2\pi$}. The formula for $\Omega$ is now
$$
\Omega(x,y,z) = \sin\left( \frac x2 \right) + \sin \left( \frac y2 \right) - \sin \left( \frac z2 \right) + \sin \left( \frac{x+y-z}{2} \right)
$$
which reduces, after using the trigonometric sum-to-product formulas, to
\begin{equation}
\label{formulaomega}
\Omega(x,y,z) = 2 \sin \left( \frac{x-z}{4} \right) \cos \left( \frac{x+z}{4} \right) + 2 \cos \left( \frac{x-z}{4} \right) \sin \left( \frac{x + 2y-z}{4} \right).
\end{equation}

\medskip

\noindent \underline{Case 2.1: $x+y-z \leq 0$.} Note that this implies that
\begin{equation}
\label{hedgehog}
-\frac \pi 2 \leq \frac{x-z}{4} \leq 0 \qquad \mbox{and} \qquad
-\frac{\pi}{2} \leq \frac{x+2y-z}{4} \leq \frac{\pi}{2}.
\end{equation}
Using the formula~\eqref{formulaomega}, $(x,y,z)$ is a zero of $\Omega$ if and only if
\begin{equation}
\label{omegazero}
 \sin \left( \frac{x + 2y-z}{4} \right) = \tan \left( \frac{z-x}{4} \right) \cos \left( \frac{x+z}{4} \right).
\end{equation}
With the help of~\eqref{boundtan} and~\eqref{hedgehog}, the sin function can be inverted to get
$$
\frac{x + 2y-z}{4} = \arcsin \left[ \tan \left( \frac{z-x}{4} \right) \cos \left( \frac{x+z}{4} \right) \right],
$$
which can also be written as $y = h(x,z)$. This shows that any solution of $\Omega(x,y,z) = 0$ with $x+y-z \leq 0$ is of the form $y = h(x,z)$.

Conversely, if $x \leq z$, we want to check that there exists a solution (in $y$) of $\Omega(x,y,z) = 0$ with $x+y-z \leq 0$. Thus, $y$ is restricted to belong to satisfy this inequality, and to belong to $[0,2\pi]$, or in other words, $0 \leq y \leq z-x$. But $\Omega$ changes sign between the two endpoints: indeed, by \eqref{star2} \eqref{star3} \eqref{hedgehog},
\begin{align*}
& \Omega(x,y=0,z) = 2 \sin \left( \frac{x-z}{4} \right) \left[ \cos \left( \frac{x+z}{4} \right) + \cos \left( \frac{x-z}{4} \right) \right] \leq 0 \\
& \Omega(x,y=z-x,z) = 2 \sin \left( \frac{x-z}{4} \right) \left[ \cos \left( \frac{x+z}{4} \right) - \cos \left( \frac{x-z}{4} \right) \right] \geq 0.
\end{align*}
Thus, $y \mapsto \Omega(x,y,z)$ has at least a zero, which was the desired statement.

 \medskip

\noindent \underline{Case 2.2: $x+y-z \geq 2\pi$.} This implies that 
$$
0 \leq \frac{x-z}{4} \leq \frac{\pi} 2 \qquad \mbox{and} \qquad
\frac{\pi}{2} \leq \frac{x+2y-z}{4} \leq \frac{3\pi}{2}.
$$ 
As a consequence, inverting the sin function in~\eqref{omegazero} with the help of~\eqref{boundtan} gives
$$
\frac{x + 2y-z}{4}  = \pi - \arcsin \left[ \tan \left( \frac{z-x}{4} \right) \cos \left( \frac{x+z}{4} \right) \right],
$$ 
which is equivalent to $y = h(x,z) + 2\pi$.

Conversely, if $x \geq z$, we need to check that there exists a solution (in $y$) of $\Omega(x,y,z) = 0$ with $x+y-z \geq 2\pi$. Besides this inequality, the variable $y$ is constrained to belong to $[0,2\pi]$; in other words, $y$ ranges in $[2\pi + z - x, 2\pi]$. There remains to check that $\Omega$ changes sign between these two endpoints; this is the case since
\begin{align*}
& \Omega(x,y=2\pi,z) = 2 \sin \left( \frac{x-z}{4} \right)  \left[ \cos \left( \frac{x+z}{4} \right) - \cos \left( \frac{x-z}{4} \right) \right] \leq 0 \\
& \Omega(x,y=2\pi+z-x,z)= 2 \sin \left( \frac{x-z}{4} \right) \left[ \cos \left( \frac{x+z}{4} \right) + \cos \left( \frac{x-z}{4} \right) \right] \geq 0.
\end{align*}
\end{proof}

\begin{lemma}[Integration on the resonant manifold with $p_2$ as integration variable]
\label{lemmaintegrationp2}
For a test function $\varphi$,
$$
\int_{\mathbb{T}^2} \delta( \Omega(x,y,z)) \varphi(x,y,z) \,dy \,dz 
= \int_0^{2\pi} \varphi(x,h(x,z),z) \frac{2 \dd z}{\sqrt{F_+(x,z)}},
$$
(considering $\varphi$ as periodic in $y,z$) where
$$
F_+(x,z) = \left[ \cos \left(\frac{x}{2} \right) + \cos \left(\frac{z}{2} \right) \right]^2 + 4 \sin \left(\frac{x}{2} \right) \sin \left(\frac{z}{2} \right).
$$
\end{lemma}

\begin{proof}
Since $h(x,z)$ is the unique zero (in $y$) of $\Omega(x,y,z)$, 
$$
\delta(\Omega(x,y,z)) \dd y = \frac{1}{|\partial_y \Omega(x,h(x,z),z)|} \delta(y - h(x,z)).
$$
The derivative of $\Omega$ is
$$
\partial_y \Omega(x,y,z) = \frac{1}{2} \left[ \cos \left( \frac{y}{2} \right) + \cos \left( \frac{x+y-z}{2} \right) \right] =  \cos \left( \frac{x+2y-z}{4} \right)  \cos \left( \frac{x-z}{4} \right).
$$
Evaluating this function at $y = h(x,z)$ and using the definition of $h$, we find
\begin{align*}
| \partial_y \Omega(x,h(x,z),z) | & = \cos \left( \frac{x-z}{4} \right) \sqrt{ 1 - \sin \left( \frac{x+2h-z}{4} \right)^2 } \\
& =  \cos \left( \frac{x-z}{4} \right)  \sqrt{ 1 - \tan \left( \frac{x-z}{4} \right)^2 \cos \left( \frac{x+z}{4} \right)^2 } \\
& = \sqrt{ \cos \left( \frac{x-z}{4} \right)^2 - \sin \left( \frac{x-z}{4} \right)^2 \cos \left( \frac{x+z}{4} \right)^2 }.
\end{align*}
There remains to check that
$$
\cos \left( \frac{x-z}{4} \right)^2 - \sin \left( \frac{x-z}{4} \right)^2 \cos \left( \frac{x+z}{4} \right)^2 = \frac 14 \left[ \cos \left(\frac{x}{2} \right) + \cos \left(\frac{z}{2} \right) \right]^2 + \sin \left(\frac{x}{2} \right) \sin \left(\frac{z}{2} \right).
$$
This can be seen as follows: starting from the expression on the left-hand side, use the formulas $2 \cos^2 x = 1 + \cos(2x)$ and $2 \sin^2 x = 1 - \cos(2x)$ and then expand the resulting expression using that $\cos(a+b) = \cos a \cos b - \sin a \sin b$.
\end{proof}

\begin{lemma}[Integration on the resonant manifold with $p_1$ as integration variable]
\label{lemmaintegrationp1}
For a test function $\varphi$, periodic in $\mathbb{T}^2$, 
$$
\int_{\mathbb{T}^2} \delta( \Omega(x,y,z)) \varphi(x,y) \, \dd y \, \dd z 
= 2 \int \varphi(x,y) \frac{\mathbb{1}_{F_-(x,y)>0}}{\sqrt{F_-(x,y)}} \dd y,
$$
where
$$
F_-(x,z) = \left[ \cos \left(\frac{x}{2} \right) + \cos \left(\frac{z}{2} \right) \right]^2 - 4 \sin \left(\frac{x}{2} \right) \sin \left(\frac{z}{2} \right).
$$
\end{lemma}
\begin{proof} 
Given Lemma \ref{lemmaintegrationp2}, it suffices to show that 
$$ \int_{[0,2\pi]} \varphi(x,y) \frac{ \mathbb{1}_{F_-(x,y)>0}}{\sqrt{F_-(x,y)}} \dd y  =  
\int_{[0,2\pi]} 
\varphi(x,h(x,z)) \frac{ \dd z }{\sqrt{F_+(x,z)}} $$
 if one of these integrals is absolutely convergent. 
We basically want to perform a change of variables for fixed $x$, to $y=h(x,z)$. 
We first notice that from Lemma \ref{lemmaintegrationp2} and due to the relation
$$ \partial_2 \Omega\vert_{Z_-}(x, h(x,z),z) \partial_z h(x,z) +\partial_z \Omega\vert_{Z_-}(x, h(x,z),z) =0, $$
we get 
$$ \vert \partial_z \Omega\vert_{Z_-}(x, h(x,z),z) \vert  = \vert \partial_z h(x,z) \vert \frac{1}{2} \sqrt{F_+(x,z)}. $$
We will now show that for $z$ so that $h(x,\cdot)$ is locally invertible, it holds that
$$\vert \partial_z \Omega\vert_{Z_-}(x, y,z) \vert  = \frac{1}{2}\sqrt{F_-(x,y)}$$
which finishes the claim. We are using the following ingredients:
\begin{itemize}
\item the fact that in order to construct all possible local inverse functions $\tilde{h}(y,x)$, of $h(x,\cdot)$, it is enough to find all $z$'s so that (for fixed $x,y$), it holds $(x,y,z) \in Z_-$ and $\Omega\vert_{Z_-} =0$. 
\item We then use the formula for $\Omega\vert_{Z_-}$: 
$$\Omega\vert_{Z_-} (x,y,z) = 2 \left( \sin \left(\frac{x+y}{4}\right)\cos \left(\frac{x-y}{4}\right) + \cos \left(\frac{x+y}{4}\right) \sin \left(\frac{x+y-2z}{4}\right) \right), $$
which yields after manipulations of trigonometric functions that $F_-(x,y)\geq 0$ is a necessary condition in order to have $\Omega\vert_{Z_-}=0$. Indeed one observes that if $\Omega\vert_{Z_-}=0$, then 
\begin{align} \label{eq:sol to Omega_- relation}
\sin \left( \frac{x+y-2z}{4} \right) = -\tan \left( \frac{x+y}{4}  \right)\cos \left( \frac{x-y}{4}  \right), 
\end{align}
which is valid (for $x,y,z \in \mathbb{R}$) only if $\frac{1}{4}F_-(x,y)\geq 0$,
{{since $F_-(x,y)\ge 0$ is equivalent to $1 \ge \cos^2 \left ( \frac{x-y}{4}\right ) \tan^2 \left ( \frac{x+y}{4} \right )$,}}.
\item Whenever $F_- >0$, there are exactly two solutions in $Z_-$ to the energy problem, explicitly given by $$z_\sigma= \tilde{h}_\sigma(y,x) = \frac{x+y}{2} + 2 \sigma \arcsin \left( \tan \left(\frac{x+y}{4} \right)\cos\left( \frac{x-y}{4}  \right) \right) + 2\pi \mathbb{1}_{\sigma=-1} (-1)^{\mathbb{1}_{x+y>2\pi}}  \ ,\ \sigma \in \{\pm 1\}
.$$
These also satisfy \eqref{eq:sol to Omega_- relation}.
\end{itemize}
Now we may compute the Jacobian of the change of variables in the class of these $z$ where explicit calculations yield
\begin{equation} \label{eq:Jac1}
\begin{split} 
\left\vert \partial_z \Omega_- \right\vert = \left\vert 
\cos\left(\frac{x+y}{4}\right)  \right\vert 
\left\vert 
\cos\left(\frac{x+y-2z}{4}\right)  \right\vert 
 & =  
\left\vert 
\cos\left(\frac{x+y}{4}\right)  \right\vert
\left\vert \sqrt{1-\sin^2\left(\frac{x+y-2z}{4}  \right)}
\right\vert   \\
&  = \frac{1}{2} \sqrt{F_-(x,y)}.
\end{split}
\end{equation}
Also note that for fixed $x$, there are at most two $y$'s that satisfy $F_-=0$, so the calculation in \eqref{eq:Jac1} holds up to a finite number of $y$'s.
\end{proof}

\begin{lemma} [Vanishing rate of $F_-$]
\label{lemmazerosF-} For fixed $x \in (0,2\pi)$, the function $y \mapsto F_-(x,y)$ has two zeros, $y' < y''$ such that
$$
0 < y' < 2\pi- x < y'' < 2\pi.
$$
It is negative between these zeros and furthermore,
\begin{equation} \label{eq:lowerboundon F_}
\begin{cases}
&F_-(x,y) \gtrsim (y'-y)\operatorname{sin} \left( \frac{x}{2}\right),\ \text{ on } [0, y')\\
&F_-(x,y) \gtrsim (y-y'') \operatorname{sin} \left( \frac{x}{2}\right),\ \text{ on } (y'', 2\pi]
\end{cases}
\end{equation}
\end{lemma}

\begin{proof} By symmetry of $F_-$ (namely, the identity $F_-(2\pi-x,2\pi-y) = F_-(x,y)$), it suffices to consider the case $0 \leq x \leq \pi$. We will denote 
$$
c = \cos \left( \frac x2 \right) \geq 0, \qquad s = \sin \left( \frac x2 \right) \geq 0.
$$

\smallskip

\noindent \underline{The derivatives of $F_-$.} Differentiating in $y$ gives
$$
F_2(x,y) = [\partial_y F_{{-}}](x,y) = - \sin \left( \frac y2 \right) \left( \cos \left( \frac x2 \right) + \cos \left( \frac y2 \right) \right) - 2 \sin \left( \frac x2 \right) \cos \left( \frac y2 \right).
$$
At this point, it is convenient to switch to the variable $u = \cos \left( \frac y 2 \right)$. Abusing notations, we write
$$
F_2(x,u) = F_2(x,y(u)) = - \sqrt{1 - u^2}(c+u) - 2 u \sqrt{1-c^2}.
$$
Taking further derivatives in $u$,
\begin{align*}
& \partial_u F_2(x,u) = \frac{2u^2+uc-1}{\sqrt{1-u^2}} - 2 \sqrt{1-c^2} \\
& \partial_u^2 F_2(x,u) = \frac{P(u)}{(1-u^2)^{\frac 32}}, \qquad P(u) = -2u^3+3u+c \\
& P'(u) = -6u^2 +3.
\end{align*}

\smallskip

\noindent \underline{Sign of $\partial_u^2 F_2(x,\cdot)$.} Since $P'$ has roots at $- 2^{-\frac 12}$ and $2^{-\frac 12}$, the former is a local minimum of $P$ and the latter a local maximum. Furthermore, $P(-1) = c-1 <0$, $P(0) = c>0$ and $P(1) = c+1 {>}0$. As a consequence, $P$ has a unique root $u_0<0$ in $[-1,1]$, and $P>0$ if and only if $u>u_0$; hence the same holds true for $\partial_u^2 F_2(x,\cdot)$.

\smallskip

\noindent \underline{Sign of $\partial_u F_2(x,\cdot)$.} By the previous paragraph, $\partial_u F_2(x,\cdot)$ is decreasing if $u<u_0$, and increasing if $u>u_0$. Furthermore, $\partial_u F_2(-1) = {{+}}\infty$, $\partial_u F_2(0) = -1-2s<0$ and $\partial_u F_2(1) = + \infty$. Therefore, $\partial_u F_2(x,\cdot)$ has two zeros, $u_-<0$ and $u_+>0$, and $\partial_u F_2(x,\cdot)>0$ if and only if $u<u_-$ or $u>u_+$.

\smallskip

\noindent \underline{Sign of $F_2(x,\cdot)$.} Coming back to the $y$ variable, we learn from the previous paragraph that $y\mapsto F_2(x,y)$ is increasing on $[y_-,y_+]$, with $y_-<\pi<y_+$, and decreasing otherwise ($y_{\pm}$ corresponds to $u_{\mp}$). Next, we note that $F_2(x,0) = -2 s <0$, $F_2(x,\pi) = -c <0$ and $F_2(x,2\pi) = 2 s >0$. Therefore, there exists $y_0>\pi$ such that $F_2(x,y)>0$ if and only if $y>y_0$.

\smallskip

\noindent \underline{Sign of $F_-(x,\cdot)$.} The previous paragraph tells us that $F_-(x,\cdot)$ is decreasing on $[0,y_0]$ and increasing on $[y_0,2\pi]$, with $y_0>\pi$. Since $F_-(x,0)=(1+c)^2>0$, $F_-(x,2\pi -x) = - 4s^2<0$ and $F_-(x,2\pi) = (c-1)^2 >0$, there exists $y'$ and $y''$ such that $F_-(x,\cdot) <0$ if and only if $y \in (y',y'')$ and furthermore $0<y'<2\pi-x<y''<2\pi$ and $y'<y_0<y''$.

\smallskip

\noindent \underline{Reformulation of the problem.} Writing 
$$
F_-(x,y) = 
\begin{cases}
- \int_y^{y'} F_2(x,z) \, dz & \mbox{if $y<y'$} \\
\int_{y''}^y F_2(x,z) \,dz & \mbox{if $y>y''$},
\end{cases}
$$
we see that it suffices to show that
$$
\begin{cases}
F_2(x,y) \lesssim -s & \mbox{if $y<y'$} \\
 F_2(x,y) \gtrsim s & \mbox{if $y>y''$}.
\end{cases}
$$
Since $F_2$ is increasing on $[y_-,y_+]$ and decreasing on $[y_+,2\pi]$ and since $y'<y_0<y''$,
$$
\mbox{on $[y'',2\pi]$}, \qquad F_2(x,y) \geq \min(F_2(x,y''),F_2(x,2\pi)) = \min(F_2(x,y''),2s).
$$
Similarly,
$$
\mbox{on $[0,y']$}, \qquad F_2(x,y) \leq \max(F_2(x,y'),F_2(x,0)) = \max(F_2(x,y'),-2s).
$$ 
By the two previous assertions, it suffices to show that
$$
F_2(x,y') \lesssim -s \quad \mbox{and} \quad F_2(x,y'') \gtrsim s.
$$

\smallskip

\noindent \underline{End of the proof.} We now refer to the argument in \cite{LukkarinenSpohn2008}, equations (4.12) to (4.17), a clever sequence of inequalities which we were not able to simplify.
\end{proof}

\subsection{Asymptotics of the collision frequency}

\begin{lemma}[Asymptotics of the collision frequency for zero mass]
\label{lemmaasymptoticszero}
If $\gamma=0$, $a(p)$ is a nonnegative continuous functions on $\mathbb{T}$ vanishing only at $0$. Furthermore,
$$
a(p) \sim \sin \left( \frac p 2 \right) ^{5/3} \qquad \mbox{if $p \in [0,2\pi]$}.
$$
\end{lemma}

\begin{proof}
It follows from the definition \eqref{chardonneret}, Lemma \eqref{lemmaparameterization} and Lemma \eqref{lemmaintegrationp2} that
$$
a(x) = \omega(x)^2 \int_0^{2\pi} \frac{2 \dd z}{\sqrt{ \left( \cos \left( \frac x 2 \right) + \cos \left( \frac z 2 \right)\right)^2 + 4 \sin \left( \frac x 2 \right) \sin \left( \frac z 2 \right)}}.
$$

First observe that $a(x) = a(2\pi -x)$, hence it suffices to consider the case $x \in [0,\pi]$. Second, as long as $x$ does not approach $0$, it is clear that $a(x) \sim 1$. Hence, it suffices 
to compute the equivalent of the integral as $x \to 0$. The contribution of this integral for $z < \frac{3\pi}2$ is bounded: indeed, the denominator is bounded from below since
$$
 \cos \left( \frac x 2 \right) + \cos \left( \frac z 2 \right) > \cos \left( \frac \pi 8 \right) - \cos \left( \frac \pi 4 \right) > 0 \quad \mbox{if $x \in (0,\frac \pi 4) \;\; \mbox{and} \;\; z \in (0,\frac{3\pi}2) $}.
 $$
Therefore, we restrict the integration variable to $(\frac{3\pi}2,2\pi)$ and then change variable to $z' = 2\pi - z$ to obtain the expresssion
$$
\int_0^{\frac \pi 2 } \frac{2 \dd z'}{\sqrt{ \left( \cos \left( \frac x 2 \right) - \cos \left( \frac {z'} 2 \right)\right)^2 + 4 \sin \left( \frac x 2 \right) \sin \left( \frac {z'} 2 \right)}}.
$$
Denoting $\epsilon = \sin \left( \frac x 2 \right)$ and $\sin  \left( \frac {z'} 2 \right)= s$, the expression in the square root in the denominator can now be written
\begin{align*}
& \left( \cos \left( \frac x 2 \right) - \cos \left( \frac {z'} 2 \right)\right)^2 + 4 \sin \left( \frac x 2 \right) \sin \left( \frac {z'} 2 \right) \\
& \qquad \qquad = \left( \frac{\sin \left( \frac x 2 \right)^2 - \sin \left( \frac {z'} 2 \right)^2}{\sqrt{1 - \sin \left( \frac x 2 \right)^2} + \sqrt{1 - \sin \left( \frac {z'} 2 \right)^2}} \right)^2 + 4 \sin \left( \frac x 2 \right) \sin \left( \frac {z'} 2 \right) \\
& \qquad \qquad = \left( \frac{\epsilon^2 - s^2 }{\sqrt{1-\epsilon^2} + \sqrt{1 -s^2 }} \right)^2+ 4 \epsilon s 
\end{align*}
The integral becomes
\begin{align*}
& \int_0^{\sqrt{2}/2} \left[ \left( \frac{\epsilon^2 - s^2 }{\sqrt{1-\epsilon^2} + \sqrt{1 -s^2 }} \right)^2+ 4 \epsilon s \right]^{-\frac 12} \frac{2 \dd s}{\sqrt{1 - s^2}} \\
 & \qquad \qquad \qquad \qquad \sim
  \int_0^{\sqrt{2}/2} \left[ \left( \epsilon^2 - s^2  \right)^2+ 4 \epsilon s \right]^{-\frac 12} \dd s \sim \epsilon^{-\frac 13}.
\end{align*}
Here, we used that
$$
\left( \epsilon^2 - s^2 \right)^2+ 4 \epsilon s 
\sim
\begin{cases}
 \epsilon^4 & \mbox{if $s < \epsilon^3$} \\
\epsilon s & \mbox{if $\epsilon^3 < s < \epsilon^{\frac 13} $} \\
s^4 & \mbox{if $s> \epsilon^{\frac 13}$} 
\end{cases}
$$
\end{proof}

\begin{lemma}
\label{lemmaomega2omega3}
For any $x,z \in [0,2\pi]$,
$$
\left|\sin \left( \frac{\underline{h}(x,z)}{2} \right) \sin \left( \frac{{h}(x,z)}{2} \right) \right| = \tan^2\left( \frac{z-x}{4} \right) \sin \left(\frac z 2 \right)\sin \left(\frac x 2 \right).
$$
\end{lemma}

\begin{proof}
Denoting
$$
g(x,z) = \arcsin \left( \tan\left( \frac{|z-x|}{4} \right) \cos \left( \frac{z+x}{4} \right) \right),
$$
we can write
\begin{align*}
\left| \sin \left( \frac{\underline{h}(x,z)}{2} \right) \right|
& = \left| \sin \left( g(x,z) - \frac{z-x}{4} \right) \right|
= \left| \sin \left( \frac{z-x}{4} \right) \right| \left|\sigma \cos \left( \frac{z+x}{4} \right) - \cos(g(x,z)) \right| \\
\left| \sin \left( \frac{{h}(x,z)}{2} \right) \right| & = \left| \sin \left( g(x,z) + \frac{z-x}{4} \right) \right|
= \left| \sin \left( \frac{z-x}{4} \right) \right| \left|\sigma \cos \left( \frac{z+x}{4} \right) + \cos(g(x,z)) \right|.
\end{align*}
where $\sigma = \sgn(z-x)$.

As a result,
\begin{align*}
& \left| \sin \left( \frac{\underline{h}(x,z)}{2} \right) \sin \left( \frac{{h}(x,z)}{2} \right) \right| =  \sin^2 \left( \frac{z-x}{4} \right) \left| \cos^2 \left( \frac{z+x}{4} \right) - \cos^2 (g(x,z)) \right| \\
& \qquad = 
\sin^2 \left( \frac{z-x}{4} \right) \left[ \cos^2 \left( \frac{z+x}{4} \right) - 1 + \tan^2\left( \frac{z-x}{4} \right)\cos^2 \left( \frac{z+x}{4} \right) \right] \\
&\qquad = \sin^2 \left( \frac{z-x}{4} \right)\frac{\cos^2 \left( \frac{z+x}{4} \right) - \cos^2 \left( \frac{z-x}{4} \right)}{\cos^2 \left( \frac{z-x}{4} \right)} =  \tan^2\left( \frac{z-x}{4} \right) \sin \left(\frac z 2 \right)\sin \left(\frac x 2 \right).
\end{align*}
\end{proof}

\begin{lemma}[Asymptotics of the collision frequency for non-zero mass]
\label{lemmaasymptoticsnonzero}
If $\gamma > 0$, $a(p)$ is a nonnegative continuous functions on $\mathbb{T}$ vanishing only at $0$. Furthermore,
$$
a(p) \sim \sin \left( \frac{p}{2} \right)^{5/3} \qquad \mbox{if $p \in [0,2\pi]$}.
$$
\end{lemma}

\begin{proof}
It follows from the definition \eqref{chardonneret}, Lemma \eqref{lemmaparameterization} and Lemma \eqref{lemmaintegrationp2} that
\begin{equation}\label{formulaa}
a(x) \sim \sin \left( \frac x 2 \right) \int_0^{2\pi} \frac{ \sin \left( \frac z 2 \right) \left| \sin \left( \frac {h(x,z)} 2 \right)  \sin \left( \frac {\underline{h}(x,z)} 2 \right) \right|}{\sqrt{ \left( \cos \left( \frac x 2 \right) + \cos \left( \frac z 2 \right)\right)^2 + 4 \sin \left( \frac x 2 \right) \sin \left( \frac z 2 \right)}} \dd z
\end{equation}

As in the previous lemma, it suffices to consider the case $x \to 0$.
By Lemma \ref{lemmaomega2omega3},
\begin{equation}
\label{formula1}
\left| \sin \left(\frac z 2 \right) \sin \left( \frac{\underline{h}(x,z)}{2} \right) \sin \left( \frac{{h}(x,z)}{2} \right) \right| = \tan^2 \left( \frac{z-x}{4} \right) \sin^2 \left(\frac z 2 \right)\sin \left(\frac x 2 \right).
\end{equation}
Setting $z' = 2\pi -z$, this becomes
\begin{equation}
\label{formula2}
\begin{split}
& \left| \sin \left(\frac z 2 \right) \sin \left( \frac{\underline{h}(x,z)}{2} \right) \sin \left( \frac{{h}(x,z)}{2} \right) \right| = \cot^2\left( \frac{z'+x}{4}\right) \sin^2 \left(\frac {z'} 2 \right)\sin \left(\frac x 2 \right)  \\
& \qquad \qquad = \frac{ \left[1+ \cos \left(\frac {z'} 2 \right) \cos \left(\frac x 2 \right)  - \sin \left(\frac {z'} 2 \right)\sin \left(\frac x 2 \right) \right] 
\sin^2 \left(\frac {z'} 2 \right)\sin \left(\frac x 2 \right) }{1 - \cos \left(\frac x 2 \right) \cos \left(\frac {z'} 2 \right) + \sin \left(\frac x 2 \right)  \sin \left(\frac {z'} 2 \right)}.
\end{split}
\end{equation}

\medskip
\noindent \underline{The case $0 < z < 2\pi-c_0$} (where $c_0 >0$). In that case, $\cos \left( \frac{z-x}{4} \right) \gtrsim 1$, and thus
$$
\left| \sin \left(\frac z 2 \right) \sin \left( \frac{\underline{h}(x,z)}{2} \right) \sin \left( \frac{{h}(x,z)}{2} \right) \right| \lesssim \sin \left( \frac x 2 \right).
$$
By \eqref{formula1} and the argument in Lemma \ref{lemmaasymptoticszero}, we see that the integrand in \eqref{formulaa} is $O(x)$, so that the contribution to $a(x)$ is $O(x^2)$.

\medskip

\noindent \underline{The case $z > 2\pi - c_0$.}
Letting $\epsilon = \sin \left( \frac x 2 \right)$ and $s = \sin \left( \frac {z'} 2 \right)$, formula \eqref{formula2} becomes
\begin{align*}
\left| \sin \left(\frac z 2 \right) \sin \left( \frac{\underline{h}(x,z)}{2} \right) \sin \left( \frac{{h}(x,z)}{2} \right) \right| & = \frac{\epsilon s^2 \left( 1-\epsilon s + \sqrt{1-\epsilon^2}\sqrt{1-s^2} \right) }{1 - \sqrt{1-\epsilon^2} \sqrt{1- s^2}  + \epsilon s } \\
& \sim \frac{\epsilon s^2}{\epsilon^2 + s^2} \sim
\begin{cases}
\frac{s^2}\epsilon & \mbox{if $s< \epsilon$} \\
\epsilon & \mbox{if $s > \epsilon$}
\end{cases}
\end{align*}
Combining this equivalent with the proof of Lemma \ref{lemmaasymptoticszero} gives the desired result.
\end{proof}

\begin{lemma}
\label{lemmaomega123}
For any $x,z \in [0,2\pi]$,
$$
\left| \sin \left( \frac {z} 2 \right) \sin \left( \frac {h(x,z)} 2 \right) \sin \left( \frac {\underline{h}(x,z)} 2 \right) \right| \lesssim \sin \left( \frac {x} 2 \right)
$$
\end{lemma}

\begin{proof} By Lemma \ref{lemmaomega2omega3},
$$
\left| \sin \left( \frac {z} 2 \right) \sin \left( \frac {h(x,z)} 2 \right) \sin \left( \frac {\underline{h}(x,z)} 2 \right) \right| = \tan^2 \left( \frac{z-x} 4 \right) \sin^2 \left( \frac{z} 2 \right) \sin \left( \frac{x} 2 \right).
$$
For $(x,z)$ not close to $(0,2\pi)$,
$$
\sin \left( \frac {z} 2 \right) \sin \left( \frac {h(x,z)} 2 \right) \sin \left( \frac {\underline{h}(x,z)} 2 \right) \lesssim \sin \left( \frac{x} 2 \right).
$$
For $(x,z)$ close to $(0,2\pi)$, we use the notation $\epsilon = \sin \left( \frac x 2 \right)$ and $s = \sin \left( \frac{z} 2 \right)$ to get as in the proof of Lemma \ref{lemmaasymptoticsnonzero}
$$
\sin \left( \frac {z} 2 \right) \sin \left( \frac {h(x,z)} 2 \right) \sin \left( \frac {\underline{h}(x,z)} 2 \right) \sim \frac{\epsilon s^2}{\epsilon^2 + s^2} \lesssim \epsilon = \sin\left(\frac{x}{2}\right).
$$
\end{proof}

\bibliographystyle{alpha}
\bibliography{bibliography}

\end{document}